\documentclass[11pt]{article}
\usepackage{amsfonts}

\usepackage{graphics}
\usepackage{color}
\usepackage{indentfirst}
\usepackage{cite}
\usepackage{latexsym}
\usepackage{amsmath}
\usepackage{amssymb}
\usepackage[dvips]{epsfig}
\usepackage{amscd}
\def\on{\bar\rho}
\newtheorem{theorem}{Theorem}[section]
\newtheorem{remark}{Remark}[section]
\newcommand{\thatsall}{\hfill$\Box$}

\newtheorem{lemma}[theorem]{Lemma}

\newtheorem{proposition}[theorem]{Proposition}

\newcommand{\n}{\rho}

\newcommand{\lm}{\lambda}

\renewcommand{\div}{ {\rm div }  }

\newcommand{\pa}{\partial}
\renewcommand{\r}{\mathbb{R}}

\newcommand{\ia}{\int_0^T}

\newcommand{\bt}{\begin{theorem}}
\newcommand{\bl}{\begin{lemma}}
\newcommand{\el}{\end{lemma}}
\newcommand{\et}{\end{theorem}}
\newcommand{\ga}{\gamma}

\newcommand{\curl}{{\rm curl} }

\newcommand{\de}{\delta}
\newcommand{\ve}{\varepsilon}
\newcommand{\la}{\label}

\newcommand{\ol}{\overline}

\newcommand{\bn}{\begin{eqnarray}}
\newcommand{\en}{\end{eqnarray}}
\newcommand{\bnn}{\begin{eqnarray*}}
\newcommand{\enn}{\end{eqnarray*}}

\newcommand{\bnnn}{\begin{eqnarray*}}
\newcommand{\ennn}{\end{eqnarray*}}

\newcommand{\ba}{\begin{aligned}}
\newcommand{\ea}{\end{aligned}}
\newcommand{\be}{\begin{equation}}
\newcommand{\ee}{\end{equation}}
\def\O{{\Omega }}

\def\norm[#1]#2{\|#2\|_{#1}}

\newcommand{\no}{\nonumber\\}

\newcommand{\si}{\sigma}

\def\la{\label}

\def\na{\nabla}
\def\on{\bar\n}

\makeatletter      
\@addtoreset{equation}{section}
\makeatother

\makeatletter      
\@addtoreset{equation}{section}
\makeatother       

\title{Global Classical Solutions to the  Compressible Navier-Stokes Equations with Navier-type slip Boundary Condition in 2D Bounded Domains }

\author{
  Yuebo CAO
 \\
   {\normalsize Department of Mathematics, College of Sciences,}\\
{\normalsize Shihezi University,}\\ {\normalsize
Shihezi 832003, P. R. China }}

\date{ }
\begin{document}
\maketitle

 \begin{abstract}
 We study the barotropic compressible Navier-Stokes equations   with Navier-type boundary condition in a two-dimensional simply connected bounded domain with $C^{\infty}$ boundary $\partial\Omega.$ By some new estimates  on the boundary related to the Navier-type slip boundary condition, the classical solution to the initial-boundary-value problem of this system exists globally in time provided the initial energy is suitably small even if the density  has large oscillations and contains  vacuum states. Futhermore,  we also prove that the oscillation of the density will grow unboundedly in the long run with an exponential rate provided  vacuum (even a point) appears initially. As we known, this is the first result concerning the global existence of classical solutions to the compressible Navier-Stokes equations with Navier-type slip boundary condition and the density containing vacuum initially for general 2D bounded smooth domains.
 \end{abstract}


\section{Introduction}

The viscous barotropic compressible Navier-Stokes equations for isentropic flows express the principles of conservation of mass and momentum in the absence of exterior forces:
   \be \la{a1}  \begin{cases}\n_t+{\rm div} (\n u)=0,\\
 (\n u)_t+{\rm div}(\n u\otimes u)-\mu\Delta u-(\mu+\lambda)\na {\rm
div} u +\na P(\n) =0, \end{cases}\ee
 where $(x,t)\in\Omega\times (0,T]$, $\Omega$ is a domain in $\r^{N}$, $t\ge 0$ is time,   $x$  is the spatial coordinate. $\rho\geq0, u=(u^1,\cdots,u^N)$ and $P(\rho)=a\rho^{\gamma}$$(a>0,\gamma>1)$
are the unknown fluid density, velocity and pressure, respectively. The constants $\mu$ and $\lambda$ are the shear viscosity and bulk coefficients respectively  satisfying the following physical restrictions: \be\la{h3} \mu>0,\quad \mu +  \frac{N}{2} \lambda\ge 0.
\ee

In this paper, we only consider the system \eqref{a1} with Navier-type slip boundary condition in a two-dimensional bounded domain. More precisely, assume that $\Omega\subset\r^2 $ be a connected bounded domain  with the boundary of class $C^{\infty}$, $n=(n^1,n
 ^2)$ is the unit outer normal vector on $\partial\Omega$, the system is solved subject to the given initial data
\be \la{h2} \n(x,0)=\n_0(x), \quad \n u(x,0)=\n_0u_0(x),\quad x\in \Omega,\ee
and  Navier-type slip boundary condition
\be \la{Navi} u \cdot n = 0, \,\,(D(u)\,n+ \vartheta u)_{tangential}=0 \,\,\,\text{on}\,\,\, \partial\Omega,\ee
where $D(u) = (\nabla u+(\nabla u)^{\rm tr})/2$ is the shear stress, $\vartheta$ is a scalar friction function which measures the tendency of the fluid to slip on the boundary, and $(D(u)\,n+ \vartheta u)_{tangential}$ is the projection of tangent plane of $(D(u)\,n+ \vartheta u)$ on $\partial\Omega$.

 The boundary condition \eqref{Navi}, introduced by Navier in \cite{Nclm1} and derived by Maxwell in \cite{MJC1} from the kinetic theory of
gases (see \cite{JWN1}), which indicate that the tangential "slip" velocity, rather than being zero, is
proportional to the tangential stress.

 If we parameterize $\partial\Omega$ by arc length, it is clear that
$$\frac{\partial n}{\partial\omega}=\kappa \omega,$$
where $\omega=(-n^2,n^1)$ and $\kappa$ is  the curvature of $\partial\Omega$. Notice that $u\cdot n=0$ on $\partial\Omega$, we obtain
$$0=\frac{\partial}{\partial\omega}(u\cdot n)=(D(u)\,n)\cdot\omega-\frac{1}{2}\curl u +\kappa u\cdot\omega.$$
Hence, Navier-type slip condition is equivalent to
\be \la{ch1}
u \cdot n = 0, \,\,\curl u=2(\kappa-\vartheta)u\cdot\omega \,\,\,\text{on}\,\,\, \partial\Omega.
\ee
Especially, when $\vartheta=\kappa$, the boundary condition reduces to
\be \la{ch11}
u \cdot n = 0, \,\,\curl u=0 \,\,\,\text{on}\,\,\, \partial\Omega.
\ee
And in the case $\vartheta\rightarrow \infty$, it becomes Dirichlet boundary condition.

The very first work concerning Navier-Stokes with the boundary condition \eqref{ch1} was done by Solonnikov and \v{S}\v{c}adilov \cite{Sva1} for $\vartheta=0$, where the authors proved the existence of a weak solution in $H^1$ of stationary Stokes system with Dirichlet condition on some part of the boundary and Navier-type slip condition on the other part. From then on, several studies have been made on the well-posedness of the problem. However, most of these studies focus on the Euler equations, the purpose is to solve the Euler equations by approximation of the  incompressible Navier-Stokes equations with Navier-type slip condition, we refer to \cite{L2}, \cite{hbdv3}, \cite{CMR}, \cite{XX1} and the references therein. There are few research works paying attention to the compressible Navier-stokes equation with \eqref{ch1}. In the 2D case $\Omega=(0,1)\times (0,1),$
 Vaigant $\&$ Kazhikhov \cite{Vka1} established global classical large solutions to \eqref{a1} with the boundary condition \eqref{ch11}  when $\lambda=\rho^\beta$ with $\beta>3.$ Hoff \cite{Ho3} studied the global existence of weak solutions with the Navier-type slip boundary condition on the half space in $\r^3$ provided the initial energy is suitably small. It should be noted that in \cite{Vka1, Ho3}, the initial density is strictly away from vacuum and the boundary of $\Omega$ is flat. Recently, Cai-Li \cite{cl} studied the existence and exponential growth of global classical solutions to the compressible Navier-Stokes equations with slip boundary condition in 3D bounded domains.

Assume $\Omega$ is a bounded domain in $\r^2$ with the boundary of class $C^{\infty}$, the initial total energy of (\ref{a1}) is defined as:
\be \la{c0}
C_0 \triangleq\int_{\Omega}\left(\frac{1}{2}\n_0|u_0|^2 +G(\rho_0) \right)dx,
\ee
with
\be \la{grho}
G(\rho) \triangleq\rho\int_{\bar{\rho}}^\rho\frac{P(s)-P(\bar{\rho})}{s^2}ds.
\ee
Denote
\bnn
\int fdx\triangleq\int_\Omega fdx,\,\,
\bar{f}\triangleq\frac{1}{|\Omega|}\int_\Omega f dx,
\enn

For   integer $k$ and $1\leq q<+\infty$, $W^{k,q}(\Omega)$ is the standard Sobolev spaces and  $$ W_0^{1,q}(\Omega)\triangleq\{u\in W^{1,q}(\Omega)~\text{:}~u~ \text{is equipped with zero trace on } \partial{\Omega}\}.$$

For some $s\in(0,1)$, the fractional Sobolev space $H^s(\Omega)$ is defined by
 $$ H^s(\Omega)\triangleq\left\{u\in L^2(\Omega)~\text{:} \int_{\Omega\times\Omega}\frac{|u(x)-u(y)|^2}{|x-y|^{n+2s}}dxdy<+\infty \right\},$$ which is a Banach space with the norm:
 $$\| u\|_{H^s(\Omega)}\triangleq \|u\|_{L^2(\Omega)}+\left(\int_{\Omega\times\Omega}\frac{|u(x)-u(y)|^2}{|x-y|^{n+2s}}dxdy\right)^\frac{1}{2}.$$

For simplicity, we denote $L^q(\Omega)$, $W^{k,q}(\Omega)$, $H^k(\Omega) \triangleq W^{k,2}(\Omega),  H_0^1(\Omega)\triangleq W_0^{1,2}(\Omega),$ and ${H^s(\Omega)}$ by $L^q$,  $W^{k,q}$, $H^k$, $H^1_0,$ and ${H^s}$ respectively.

Set $$H_\omega^1(\Omega)\triangleq\{f\in H^1:f\cdot n=0, ~~ {\rm curl}f=2(\kappa-\vartheta)f\cdot\omega ~\mbox{ \rm on } ~\partial \Omega\}.$$

For two $2\times 2$  matrices $A=\{a_{ij}\},\,\,B=\{b_{ij}\}$, we define
 $$ A\colon  B\triangleq \sum\limits_{i,j=1}^{2}a_{ij}b_{ji}.$$

Finally, we set $ \nabla^{\perp}f\triangleq (\frac{\partial f}{\partial x_2},-\frac{\partial f}{\partial x_1})$, $\nabla_jf\triangleq (\nabla f)^j$, $\nabla_j^{\perp}f\triangleq (\nabla^{\perp}f)^j$, when $f$ is a scalar function. For vector function $v=(v^1,v^2)$, we denote $\nabla^{\perp}v\triangleq (\nabla^{\perp}v^1,\nabla^{\perp}v^2)$, $\nabla_j^{\perp}v\triangleq (\nabla_j^{\perp}v_1,\nabla_j^{\perp}v_2),$ $\nabla_jv\triangleq (\nabla_jv_1,\nabla_jv_2),\,j=1,2$. The
material derivative of $v$ is denoted by $\dot v\triangleq v_t+u\cdot\nabla v$.

We can now state our main results, Theorem \ref{th1} and Theorem \ref{th2}, concerning existence and large-time behavior of global classical solutions to the problem  \eqref{a1}-\eqref{ch1}.
\begin{theorem}\la{th1} Let $\Omega $ be a simply connected bounded domain in $\r^2$ with $C^{\infty}$ boundary $\partial\Omega,$ $\vartheta\in H^3$.
 For some $q>2$, $s\in (\frac{1}{2},1]$ and for given positive constants $\hat\rho$, $M$, suppose that  $\vartheta\geq\kappa$ on $\partial\Omega$, and the initial data $(\n_0,u_0)$ satisfies
\be \la{dt1}   (\rho_0 ,P(\rho_0) )  \in  W^{2,q}, \quad  u_0\in H^2\cap H_\omega^1(\Omega), \ee
\be\la{dt2} 0\leq\rho_0\leq\hat{\rho},~~\|u_0\|_{H^s}\leq M, \ee
and the compatibility condition
\be\la{dt3}
-\mu\triangle u_0-(\mu+\lambda)\nabla\div u_0 + \nabla P(\rho_0) = \rho_0^{1/2}g, \ee
for some  $ g\in L^2.$
Then there exists a positive constant $\ve$ depending only on  $\mu,$ $\lambda,$ $\ga,$ $a,$ $\hat{\rho},$ $s,$ $\vartheta,$ $\Omega,$ and $M$  such that
\be\la{dt30}
C_0\leq\varepsilon, \ee
then the problem \eqref{a1}--\eqref{ch1} has a unique global classical solution $(\rho,u)$ in $\Omega\times(0,\infty)$ satisfying
\be\la{dt6}\begin{cases}
 (\rho ,P )\in C([0,T];W^{2,q} ),\\  \na u\in C([0,T];H^1 )\cap  L^\infty(\tau,T; W^{2,q}),\\
u_t\in L^{\infty} (\tau,T; H^2)\cap H^1 (\tau,T; H^1),\\   \sqrt{\n}u_t\in L^\infty(0,\infty;L^2),
\end{cases}\ee
and that for any $0<T<\infty,$
\be\la{dt5}
  \tilde{C}(T)\inf_{x\in\Omega}\rho_0(x)\le\n(x,t)\le 2\hat{\n},\quad  (x,t)\in \O\times[0,T],
\ee
some positive constant $\tilde{C}(T)$ depending only on $T,$ $\mu,$ $\lambda,$ $\ga,$  $a,$ $\hat\rho,$  $s,$ $\Omega$, $M$ and $\vartheta$.
Moreover,  for any $r\in [1,\infty)$ and $p\in [1,\infty),$ there exist positive constants $C$ and $\tilde{\eta}$ depending only  on $\mu,$  $\lambda,$  $\gamma,$ $a$, $s$,  $\hat{\rho}$, $M, \bar{\rho}_0$,  $\Omega$, $r$, and $p$   such that
\be  \la{qa1w} \left(\|\rho-\bar{\rho}_0\|_{L^r}+\|  u\|_{W^{1,p}} +\|\sqrt{\rho}\dot{u}\|^2_{L^2}\right)\leq Ce^{-\tilde{\eta} t}.\ee
\end{theorem}

Then, with the   exponential decay rate  \eqref{qa1w} at hand,  modifying slightly the proof of \cite[Theorem 1.2]{lx},
we can   establish  the following large-time behavior of the gradient of the
density when vacuum states appear initially.
\begin{theorem}\la{th2}
Under the conditions of Theorem \ref{th1}, assume further  that
there exists some point $x_0\in \Omega$ such that $\n_0(x_0)=0.$  Then the unique
global classical solution $(\n,u)$ to the   problem  \eqref{a1}-\eqref{ch1} obtained in
Theorem \ref{th1}  satisfies that for any $\tilde{r}>2,$   there exist positive constants $\hat{C}_1$ and $\hat{C}_2$ depending only  on $\mu$,  $\lambda$,  $\gamma$, $a$, $s$,   $\hat{\rho}$, $M,$ $\bar{\n}_0$, $\Omega$, and $\tilde{r}$   such that for any $t>0$,
\be\la{qa2w}\ba \|\na\n (\cdot,t)\|_{L^{\tilde{r}}}\geq \hat{C}_1 e^{\hat{C}_2 t} . \ea\ee
\end{theorem}

A few remarks are in order:
\begin{remark} Since $q>2,$ it follows from Sobolev's inequality and \eqref{dt6}$_1$  that \be\la{soh1}  \n,\na \n \in C(\bar\Omega\times [0,T]).\ee
Moreover, it also follows from \eqref{dt6}$_2$ and \eqref{dt6}$_3$ that  \be \la{soh2} u,\na u, \na^2 u, u_t \in C(\bar\Omega\times [\tau,T]),\ee due to the following simple fact that $$L^2(\tau,T;H^1)\cap H^1(\tau,T;H^{-1})\hookrightarrow C([\tau,T];L^2).$$
Finally, by \eqref{a1}$_1,$ we have \be \n_t=-u\cdot\na \n-\n\div u\in C(\bar\Omega\times [\tau,T]),\ee which together with \eqref{soh1} and \eqref{soh2} shows that the solution obtained by Theorem \ref{th1} is a classical one.
\end{remark}

\begin{remark} It seems that our Theorem \ref{th1} is the first result concerning the global existence of the compressible Navier-Stokes equations \eqref{a1} with the density containing vacuum initially for general 2D bounded smooth domains. We also remind the reader that Navier-Stokes equations \eqref{a1} with the boundary condition \eqref{ch11} is a special case in our theorems. \end{remark}

\begin{remark}
Theorem \ref{th2} implies that   the oscillation of the density will grow unboundedly in the long run with an exponential 
  rate provided  vacuum (even a point) appears initially.  This new phenomena is   somewhat surprisingly since for the  Cauchy problem where without any rate.
\end{remark}
\begin{remark} It would be interesting to study the existence and large time asymptotic
behavior of solutions for the no-slip boundary condition $u=0$ on $\partial\Omega$. This is left for the future.\end{remark}

We now comment on the analysis of this paper. Indeed, our research bases on three observations. First,
 for $v=(v^1,v^2),$ denoting the
material derivative $\dot v\triangleq v_t+u\cdot\nabla v,$    we rewrite $ (\ref{a1})_2 $ in the form
\be \la{hod1}\ba
\rho\dot{u}=\nabla F - \mu\nabla^{\perp}\curl u ,
\ea \ee with
\be \la{dt0}  \text{curl} u \triangleq \pa_1u^2-\pa_2u^1 ,\quad F\triangleq(\lambda+2\mu)\,\div u-(P-\Bar P),\ee where $F$   is called the effective viscous flux and plays an important role in our following analysis.  Combining \eqref{hod1} with the slip boundary condition \eqref{ch1} implies that one can  treat $(\ref{a1})_2$ as a Helmholtz-Wyle decomposition of $\rho\dot{u}$ which   makes it possible to estimate $\nabla F$ and $\nabla \curl u$. Then the second observation comes from the following inequality
$$\|\nabla u\|_{L^q}\leq C(\|\div u\|_{L^q}+\|\curl u\|_{L^q})\,\,\,\text{for any} \,\,\,q>1,$$
geometrically, which is called Gaffney-Friedrichs inequality when $p=2$. Thanks to \cite{vww}, it turns out that for $u\in W^{1,q}$ with $u\cdot n=0$ on $\partial\Omega$, the inequality really holds if and only if $\Omega$ is simply connected. This inequality allows us to control $\nabla u$ by means of $\div u$ and $\curl u$. Finally, since $u\cdot n=0$ on $\partial\Omega$, we have
\bnn u\cdot\nabla u\cdot n=-u\cdot\nabla n\cdot u=-\kappa |u|^2,\enn
which is the key to estimating the integrals on the boundary $\partial\Omega$ together with $\curl u=2(\kappa-\vartheta)u\cdot\omega$ on $\partial\Omega$.

The rest of the paper is organized as follows. First,  some notations, known facts and elementary inequalities needed in later analysis are collected in Section 2. Section 3 and Section 4 are devoted to deriving the necessary a priori estimates on classical solutions which can guarantee the local classical solution to be a global classical one. Finally, the main results, Theorems \ref{th1} and  \ref{th2} will be proved in Section 5.
\section{Preliminaries}\la{se2}

In this section, we review some known facts and elementary inequalities, which we will used later.
\subsection{Some known inequalities and facts}

In this subsection, we will recall some known theorems and facts, which  are frequently applied in this paper.

First, we give an important lemma about the local existence of strong and classical solutions, its proof is similar to \cite[Theorem 1.4]{hxd1}.
\begin{lemma}\la{loc1} Let $\Omega$ be as in Theorem \ref{th1}, assume that $(\n_0,u_0)$ satisfies \eqref{dt1} and \eqref{dt3}. Then there exist a small time $T>0$ and a unique strong solution $(\n,u)$ to the problem \eqref{a1}-\eqref{ch1} on $\Omega\times(0,T]$ satisfying for any $ \tau\in(0,T),$
\be\nonumber\begin{cases}
 (\rho,P )\in C([0,T];W^{2,q} ),\\  \na u\in C([0,T];H^1 )\cap  L^\infty(\tau,T; W^{2,q}),\\
u_t\in L^{\infty} (\tau,T; H^2)\cap H^1 (\tau,T; H^1),\\   \sqrt{\n}u_t\in L^\infty(0,T;L^2) .
\end{cases}\ee \end{lemma}




Next, the following well-known Gagliardo-Nirenberg's inequality (see \cite{nir})
  will be used later.
\begin{lemma}
[Gagliardo-Nirenberg]\la{l1} Assume that $\Omega$ is a bounded Lipschitz domain in $\r^2$. For  $p\in [2,\infty),q\in(1,\infty), $ and
$ r\in  (2,\infty),$ there exists some generic
 constant
$C>0$ which may depend  on $p,\,\,q,$ and $r$ such that for   $f\in H^1({\O }) $
and $g\in  L^q(\O )\cap W^{1,r}(\O), $    we have
\be\la{g1}\|f\|_{L^p(\O)}\le C_1 \|f\|_{L^2}^{\frac{2}{p}}\|\na
f\|_{L^2}^{1-\frac{2}{p}}+C_2\|f\|_{L^2} ,\ee
\be\la{g2}\|g\|_{C\left(\ol{\O }\right)} \le C_1
\|g\|_{L^q}^{q(r-2)/(2r+q(r-2))}\|\na g\|_{L^r}^{2r/(2r+q(r-2))} + C_2\|g\|_{L^2}.
\ee
Moreover, if $f\cdot n|_{\partial\Omega}=0,\,\,\,g\cdot n|_{\partial\Omega}=0,$ or $\bar{f}=0$, $\bar{g}=0$, then the constant $C_2=0.$
\end{lemma}

Next,   to get the
uniform (in time) upper bound of the density $\n,$ we need the following Zlotnik  inequality.
\begin{lemma}[\cite{zl1}]\la{le1}   Let the function $y$ satisfies
\bnn y'(t)= g(y)+b'(t) \mbox{  on  } [0,T] ,\quad y(0)=y^0, \enn
with $ g\in C(R)$ and $y,b\in W^{1,1}(0,T).$ If $g(\infty)=-\infty$
and \be\la{a100} b(t_2) -b(t_1) \le N_0 +N_1(t_2-t_1)\ee for all
$0\le t_1<t_2\le T$
  with some $N_0\ge 0$ and $N_1\ge 0,$ then
\bnn y(t)\le \max\left\{y^0,\hat{\zeta} \right\}+N_0<\infty
\mbox{ on
 } [0,T],
\enn
  where $\hat{\zeta} $ is a constant such
that \be\la{a101} g(\zeta)\le -N_1 \quad\mbox{ for }\quad \zeta\ge \hat{\zeta}.\ee
\end{lemma}

Consider the Lam\'{e}'s system
\be\la{cxtj1}\begin{cases}
-\mu\Delta u-(\lambda+\mu)\nabla\div u=f, \,\, &x\in\Omega, \\
u\cdot n=0,\,\,\curl u=2(\kappa-\vartheta)u\cdot\omega,\,\,&x\in\partial\Omega,
\end{cases} \ee
where $u=(u_1,u_2),\,\,f=(f_1,f_2)$, $\Omega$ is a bounded domain in $\r^2,$ and $\mu,\lambda$ satisfy the condition \eqref{h3}.

The following estimate is standard.
\begin{lemma}  [\cite{adn}] \la{zhle}
Let $u$ be a solution of the Lam\'{e}'s equation \eqref{cxtj1}, there exists a positive constant $C$ depending only on $\lambda,\,\mu,\,q,\,\,k$ and $\Omega$ such that

(1) If $f\in W^{k,q}$ for some $q\in(1,\infty),\,\, k\geq0,$ then $u\in W^{k+2,q}$ and
$$\|u\|_{W^{k+2,q}}\leq C(\|f\|_{W^{k,q}}+\|u\|_{L^q}),$$

(2) If $f=\nabla g$ and $g\in W^{k,q}$ for some $q\in(1,\infty),\,\,k\geq0,$ then $u\in W^{k+1,q}$ and
$$\|u\|_{W^{k+1,q}}\leq C(\|g\|_{W^{k,q}}+\|u\|_{L^q}).$$
\end{lemma}

The following conclusion is given in \cite{vww,CANEHS,dm}.
\begin{lemma}   \la{xzle}
Let $1<q<+\infty,$ $\Omega$ is a bounded domain in $\r^2$ with Lipschitz boundary $\partial\Omega.$ For $v\in W^{1,q}$, if $\Omega$ is simply connected and $v\cdot n=0$ on $\partial\Omega$, then it holds that
$$\|\nabla v\|_{L^q}\leq C(\|\div v\|_{L^q}+\|\curl v\|_{L^q}).$$
\end{lemma}
\begin{remark}
When it comes to an exterior domain, the conclusion of Lemma \ref{xzle} is no longer true  since the first Betti number does not vanish. In fact, in \cite{ILN}, the authors proved that there exists a unique classical solution of the problem
\be\nonumber\begin{cases}
 \div v=0, \,\,\text{in}\,\,\, \Omega^c\\ \curl u=0 \,\,\text{in}\,\,\, \Omega^c,\\
v\cdot n=0 \,\,\text{on}\,\,\, \partial\Omega,\\ |v|\rightarrow 0 \,\,\text{as}\,\,\, |x|\rightarrow \infty,\\ \int_{\partial\Omega} v ds=1.
\end{cases}\ee
where $\Omega^c$ is the complementary set of a simply connected bounded domain $\Omega$ in $\r^2$.
\end{remark}

Finally, the following Beale-Kato-Majda type inequality, which was first proved in \cite{bkm,kato} when $\div u\equiv 0,$ and improved in \cite{hlx}, we give a similar result with respect to slip boundary condition to estimate $\|\nabla u\|_{L^\infty}$ and $\|\nabla\rho\|_{L^q}(q\geq2).$
\begin{lemma}   \la{le9}  For $2<q<\infty,$ assume that $u\cdot n=0$ and $\curl u=2(\kappa-\vartheta)u\cdot\omega,\,\,\,\nabla u\in W^{1,q},$ there is a
constant  $C=C(q)$ such that  the following estimate holds
\bnn \la{ww7}\ba
\|\na u\|_{L^\infty}\le C\left(\|{\rm div}u\|_{L^\infty}+\|\curl u\|_{L^\infty} \right)\log(e+\|\na^2u\|_{L^q})+C\|\na u\|_{L^2} +C . \ea\enn
\end{lemma}
\begin{proof}
It follows from \cite{soln1} and \cite{soln2} that $u$ can be represented in the form
\bnn\ba u^{i}&=\int G_{i,\cdot}(x,y)\cdot(\mu\Delta_y u+(\lambda+\mu)\nabla_y\div_y u)dy\\
 &\triangleq\int G_{ij}(x,y)(\mu\Delta_y u^j+(\lambda+\mu)\partial_{y^j}\div_y u)(y)dy,\ea\enn
where, $G=\{G_{ij}\}$ with $G_{ij}=G_{ij}(x,y)\in C^{\infty}(\Omega\times\Omega\backslash D)$, $D\equiv\{(x,y)\in\Omega\times\Omega:\,x=y\}$, is Green matrix of the Lam\'{e}'s system
\eqref{cxtj1} which satisfies that for every multi-indexes $\alpha=(\alpha_1,\alpha_2)$ and $\beta=(\beta_1,\beta_2)$, there is a constant $C_{\alpha,\beta}$ such that for all
$(x,y)\in\Omega\times\Omega\backslash D,$ and $i,j=1,2,$
$$|\partial_{x}^{\alpha}\partial_{y}^{\beta}G_{ij}(x,y)|\leq C_{\alpha,\beta}|x-y|^{-|\alpha|-|\beta|},$$
where $|\alpha|=\alpha_1+\alpha_2$ and $|\beta|=\beta_1+\beta_2$.

Therefore,
\be\la{zhbc1}\ba
u^{i}(x)&=(\lambda+2\mu)\int G_{i,\cdot}(x,y)\cdot\nabla_y\div u(y)dy\\&\quad-\mu\int G_{i,\cdot}(x,y)\cdot\nabla_y^\perp(\curl u(y)-2(\kappa-\vartheta)u(y)\cdot\omega)dy\\
&\quad-\mu\int  G_{i,\cdot}(x,y)\cdot\nabla^{\perp}_y(2(\kappa-\vartheta) u(y)\cdot \omega) dy\\
&\triangleq\sum_{j=1}^{3}Q_j^i,
\ea\ee
since $\Delta u=\nabla\div u-\nabla^\bot\curl u.$

Let's first estimate the term $\nabla Q_1$. Let $\delta\in(0,1]$ be a constant to be chosen and introduce a cut-off function $\eta_{\delta}(x)$ satisfying $\eta_{\delta}(x)=1$ for $|x|<\delta,$ $\eta_{\delta}(x)=0$ for $|x|>2\delta,$ and $|\nabla\eta_{\delta}(x)|<C\delta^{-1}.$ Then $\nabla Q_1^i$ can be rewritten as
\be\la{zhbc2}\ba
\nabla Q_{1}^i&=(\lambda+2\mu)\int\eta_{\delta}(|x-y|)\,\nabla_x G_{i,\cdot}(x,y)\cdot\nabla_y\div u(y)dy\\
&\quad+(\lambda+2\mu)\int\nabla_y\eta_{\delta}(|x-y|)\cdot\nabla_x G_{i,\cdot}(x,y)\,\div_y u(y)dy \\
&\quad-(\lambda+2\mu)\int(1-\eta_{\delta}(|x-y|))\,\nabla_x\div_y G_{i,\cdot}(x,y)\,\div u(y)dy \\
&\triangleq(\lambda+2\mu)\sum_{k=1}^{3}\tilde{I}_k.
\ea\ee
due to $G_{i,\cdot}(x,y)\cdot n=0$ on $\partial\Omega$.

Now we estimate $\tilde{I}_k,\,\,k=1,2,3.$
\be\la{zhbc3}\ba
|\tilde{I}_1|&\leq C\|\eta_{\delta}(|x-y|)\nabla_x G_{i,\cdot}(x,y)\|_{L^{q/(q-1)}}\|\nabla^{2}u\|_{L^q}\\
&\leq C\left(\int_0^{2\delta}r^{-q/(q-1)}rdr\right)^{(q-1)/q}\|\nabla^{2}u\|_{L^q}\\
&\leq C\delta^{(q-2)/q}\|\nabla^{2}u\|_{L^q},
\ea\ee
\be\la{zhbc4}\ba
|\tilde{I}_2|&=\left|\int\nabla\eta_{\delta}(|x-y|)\cdot\nabla_x G_{i,\cdot}(x,y)\div u(y)dy\right|\\
&\leq C\int|\nabla\eta_{\delta}(y)\cdot\nabla_x G_{i,\cdot}(x,y)|dy\|\div u\|_{L^{\infty}}\\
&\leq C\int_\delta^{2\delta}\delta^{-1}r^{-1}rdr\|\div u\|_{L^{\infty}}\\
&\leq C\|\div u\|_{L^{\infty}},
\ea\ee
\be\la{zhbc5}\ba
|\tilde{I}_3|&=\left|\int(1-\eta_{\delta}(|x-y|))\nabla_x\div G_{i,\cdot}(x,y)\div u(y)dy\right|\\
&\leq C\left(\int_{\delta\leq|x-y|\leq1}+\int_{|x-y|>1}\right)|\nabla_x\div_y G_{i,\cdot}(x,y)||\div u(y)|dy\\
&\leq C\int_\delta^{1}r^{-2}rdr\|\div u\|_{L^{\infty}}+C\left(\int_1^{\infty}r^{-4}rdr\right)^{\frac{1}{2}}\|\div u\|_{L^2}\\
&\leq -C\ln\delta\|\div u\|_{L^{\infty}}+C\|\nabla u\|_{L^2}.
\ea\ee
Hence,
\be\la{zhbc61}\ba
\|\nabla Q_1\|_{L^{\infty}}\leq C\left(\delta^{(q-2)/q}\|\nabla^{2}u\|_{L^q}+(1-\ln\delta)\|\div u\|_{L^{\infty}}+\|\nabla u\|_{L^2}\right).
\ea\ee
Since $\curl u=2(\kappa-\vartheta)u\cdot\omega$ on $\partial\Omega$, for $\nabla Q_2^i$, we have
\bnn\ba
\nabla Q_2^i&=-\mu\int\eta_{\delta}(|x-y|)\,\nabla_x G_{i,\cdot}(x,y)\cdot\nabla_y^\perp(\curl u(y)-2(\kappa-\vartheta)u(y)\cdot\omega)dy\\
&\quad-\mu\int\nabla^\perp\eta_{\delta}(|x-y|)\cdot\nabla_x G_{i,\cdot}(x,y)(\curl u(y)-2(\kappa-\vartheta)u(y)\cdot\omega)dy \\
&\quad-\mu\int(1-\eta_{\delta}(|x-y|))\,\nabla_x \curl_y G_{i,\cdot}(x,y)(\curl u(y)-2(\kappa-\vartheta)u(y)\cdot\omega)dy. \\
\ea\enn
A similar calculation leads to
\be\la{zhbc62}\ba
\|\nabla Q_2\|_{L^{\infty}}\leq C\left(\delta^{(q-2)/q}\|\nabla^{2}u\|_{L^q}+(1-\ln\delta)\|\curl u\|_{L^{\infty}}+\|\nabla u\|_{L^2}\right),
\ea\ee
and
\be\la{zhbc64}\ba
\|\nabla Q_3\|_{L^{\infty}}\leq C\left(\delta^{(q-2)/q}\|\nabla^{2}u\|_{L^q}+(1-\ln\delta)\| \|u\|_{L^{\infty}}+\| \nabla u\|_{L^2}\right),
\ea\ee
It follows from \eqref{zhbc61}, \eqref{zhbc62}, \eqref{zhbc64} and Lemma \ref{xzle} that
\be\la{zhbc6}\ba
\|\nabla u\|_{L^{\infty}}\leq C\left(\delta^{(q-2)/q}\|\nabla^{2}u\|_{L^q}+(1-\ln\delta)(\|\div u\|_{L^{\infty}}+\|\curl u\|_{L^{\infty}})+\|\nabla u\|_{L^2}\right).
\ea\ee
Set $\delta=\min\{1,\|\nabla^{2}u\|_{L^q}^{-q/(q-2)}\}.$  Then \eqref{zhbc6} becomes
\bnn\ba
\|\nabla u\|_{L^{\infty}}\leq C\left(1+\ln(e+\|\nabla^{2}u\|_{L^q})(\|\div u\|_{L^{\infty}}+\|\curl u\|_{L^{\infty}})+\|\nabla u\|_{L^2}\right).
\ea\enn
This completes the proof.
\end{proof}
\subsection{Estimates for $F$, $\curl u$ and $\nabla u$}
Set $F=(\lambda+2\mu)\text{div}u-(P-\bar{P})$, which is called the viscous flux, it  plays an important role in our following analysis. For $F$, $\curl u$ and $\nabla u$, we give the following conclusions, which will be used frequently later.
\begin{lemma} \la{le3}
  Let $(\n,u)$ be a smooth solution of
   \eqref{a1} with slip boundary condition \eqref{ch1},
    then for   $2\leq p<+\infty$ there exists a positive constant $C$ depending only on $p,\,\,q,\,\,\mu,\,\,\lambda,$ and $\Omega$ such that
\be\la{tdu1}
\|\nabla u\|_{L^p}\le C(\|\div u\|_{L^p}+\|\curl u\|_{L^p}),\\
\ee
\be\la{hh19}\ba
\|{\curl u}\|_{L^p}\leq C(\|\rho\dot{u}\|_{L^2}+\|\nabla u\|_{L^2}),\ea\ee
\be\la{h19}\ba
\|{\nabla F}\|_{L^p} +\|{\nabla \curl u}\|_{L^p}\leq C(\|\rho\dot{u}\|_{L^p}+ \|\nabla u\|_{L^p}),\ea\ee
\be\la{h20}\ba
\|F\|_{L^p}&\le C\|\rho\dot{u}\|_{L^2}^{1-\frac{2}{p}}(\|\nabla u\|_{L^2} +\|P-\bar{P}\|_{L^2})^{\frac{2}{p}}+C(\|\nabla u\|_{L^2}+\|P-\bar{P}\|_{L^2}),
\ea\ee
\be\la{h18}\ba
\|\nabla u\|_{L^p}&\le C\|\rho\dot{u}\|_{L^2}^{1-\frac{2}{p}}(\|\nabla u\|_{L^2}
+\|P-\bar{P}\|_{L^2})^{\frac{2}{p}}+C(\|\nabla u\|_{L^2}+\|P-\bar{P}\|_{L^p}).
\ea\ee
Moreover,
\be\la{hh20}\ba
\|F\|_{L^p}\le C(\|\rho\dot{u}\|_{L^2}+\|P-\bar{P}\|_{L^2}+\|\nabla u\|_{L^2}),
\ea\ee
\be\la{tb101}\ba
\|\nabla^{2} u\|_{L^p}&\le C(\|\rho\dot{u}\|_{L^p}+\|P-\bar{P}\|_{W^{1,p}}+\|\nabla u\|_{L^2} ),
\ea\ee
\be\la{tb102}\ba
\|\nabla^{3} u\|_{L^p}&\le C(\|\rho\dot{u}\|_{W^{1,p}}+\|P-\bar{P}\|_{W^{2,p}}+\|\nabla u\|_{L^2}).
\ea\ee
\end{lemma}
\begin{proof}
Since $u\cdot n=0$ on $\partial\Omega$ and $\Omega$ is simply connected, we can easily get \eqref{tdu1} from Lemma \ref{xzle} .

Note that
\bnn\begin{cases}
\mu\Delta\curl u=\curl(\rho\dot{u}),~ &x\in\Omega,\\ \curl u=2(\kappa-\vartheta)u\cdot\omega,~&x\in\partial\Omega.
\end{cases}\enn
where $\vartheta$ is given by \eqref{ch1}.

The standard $L^p$-estimate of elliptic equations shows that
\be\la{hz1}\ba
\|\curl u\|_{W^{1,p}}\leq C(\|\rho\dot{u}\|_{L^p}+\|\nabla u\|_{L^p}).
\ea\ee

So we have
\be\la{hz2}\ba
\|{\nabla F}\|_{L^p}\le C(\|\nabla^\bot\curl u\|_{L^p}+ \|\rho\dot{u}\|_{L^p})\le C(\|\rho\dot{u}\|_{L^p}+\|\nabla u\|_{L^p}),
\ea\ee
since $\nabla F=\mu\nabla^\bot\curl u+\rho\dot{u}$.

Together with \eqref{hz1} and \eqref{hz2} give \eqref{h19}. \eqref{hh19} is a straightforward result of Gagliardo-Nirenberg's inequality.
It follows from \eqref{g1} and \eqref{h19}  that
\bnn\ba
\|F\|_{L^p}&\le C\|F\|_{L^2}^{\frac{2}{p}}\|\nabla F\|_{L^2}^{1-\frac{2}{p}} \\
&\le C(\|\rho\dot{u}\|_{L^2}+\|\nabla u\|_{L^2})^{1-\frac{2}{p}}(\|\nabla u\|_{L^2}
+\|P-\bar{P}\|_{L^2})^{\frac{2}{p}}\\
&\le C \|\rho\dot{u}\|_{L^2}^{1-\frac{2}{p}}(\|\nabla u\|_{L^2}
+\|P-\bar{P}\|_{L^2})^{\frac{2}{p}}+C(\|\nabla u\|_{L^2}+\|P-\bar{P}\|_{L^2}),
\ea\enn

which also implies that \eqref{hh20} is true by Young's inequality.
\be\la{hz3}\ba
\|\curl u\|_{L^p}&\le C\|\curl u\|_{L^2}^{\frac{2}{p}}\|\nabla\curl u\|_{L^2}^{1-\frac{2}{p}}+C\|\nabla u\|_{L^2}\\
&\le C\|\nabla u\|_{L^2}^{\frac{2}{p}}\|\rho\dot{u}\|_{L^2}^{1-\frac{2}{p}}+C\|\nabla u\|_{L^2},
\ea\ee
due to \eqref{g1} and \eqref{hz1}.

Since $u\cdot n=0$ on $\partial\Omega$, we deduce from Lemma \ref{xzle}, \eqref{h20} and  \eqref{hz3} that
\bnn \ba
\|\nabla u\|_{L^p}&\le C(\|\div u\|_{L^p}+\|\curl u\|_{L^p})\\
&\leq C(\|F\|_{L^p}+\|\curl u\|_{L^p}+\|P-\bar{P}\|_{L^p}) \\
&\le C\|\rho\dot{u}\|_{L^2}^{1-\frac{2}{p}}(\|\nabla u\|_{L^2}
+\|P-\bar{P}\|_{L^2})^{\frac{2}{p}} + C(\|\nabla u\|_{L^2}+\|P-\bar{P}\|_{L^p}).
 \ea\enn

For the Lam\'{e}'s system
\be\la{lames}\begin{cases}
-\mu\Delta u-(\lambda+\mu)\nabla\div u=-\rho\dot{u}-\nabla(P-\bar{P}), \,\, &x\in\Omega, \\
u\cdot n=0\,\,and\,\,\curl u=2(\kappa-\vartheta)u\cdot\omega,\,\,&x\in\partial\Omega,
\end{cases} \ee
by Lemma \ref{zhle},  \eqref{h18} and Poincar\'{e}'s inequality, we obtain
\bnn\ba
\|\nabla^{2} u\|_{L^p}&\le C(\|\rho\dot{u}\|_{L^p}+\|P-\bar{P}\|_{W^{1,p}}+\|\nabla u\|_{L^2} ) ,
\ea\enn
and
\bnn\ba
\|\nabla^{3} u\|_{L^p}&\le C(\|\rho\dot{u}\|_{W^{1,p}}+\|P-\bar{P}\|_{W^{2,p}}+\|\nabla u\|_{L^2}).
\ea\enn
\end{proof}
\section{\la{se3} A priori estimates (I): lower order estimates}

This section contains some necessary a priori bounds for smooth solutions to the initial-boundary value problem (\ref{a1})-(\ref{ch1}) .
 It is assumed, without further mentioned it, that $\Omega$ be a bounded domain in $\r^2$ with $C^{\infty}$ boundary $\partial\Omega$, even though some results hold under weaker conditions on the regularity of the boundary of the domain. $T>0$ be a fixed time and $(\rho,u)$ be
the smooth solution to (\ref{a1})-(\ref{ch1})  on
$\Omega \times (0,T]$  with smooth initial
data $(\n_0,u_0)$ satisfying $u_0\in H^s$ for some $s\in(0,1)$ and $0\leq\rho_0\leq \hat{\rho}$.

In fact, $\omega=(-n^2,n^1)$ is the unit tangent vector of $\partial\Omega$. Since $u\cdot n=0$ on $\partial\Omega$, it is obvious that $u$ is parallel to $\omega$. Thus we have $u=(u\cdot\omega)\, \omega$.

A simple computation shows that
\be\la{bdd1}u\cdot\nabla u\cdot n=-u\cdot\nabla n\cdot u, \,\,x\in\partial\Omega,\ee due to $u\cdot n=0$ on $\partial\Omega$.
Set $\kappa=\kappa(x)\triangleq \omega\cdot\nabla n\cdot \omega $, geometrically, which is in fact the curvature at the point $x\in\partial\Omega$. So we have
\be\la{bdd2}u\cdot\nabla u\cdot n=-\kappa|u|^2=-\kappa (u\cdot\omega)^2\ee
on the boundary $\partial\Omega$.

One can extend the functions $n$ and $\kappa$ to $\Omega$ such that $n,\kappa\in C^3(\bar{\Omega})$, in the following discussion we still denote them by $n$ and $\kappa$ respectively.

First, we give the following estimates, which depend on the boundary condition $u\cdot n=0$ on $\partial\Omega$.
\begin{lemma}\la{uup1}Suppose $(\n,u)$ is a smooth solution of
   (\ref{a1}) with $u\cdot n=0$ on $\partial\Omega$. Then there exist a positive constant $C$ depending only on $p$ and $\Omega$ such that
\be\la{tb90}
\ba\|\dot{u}\|_{L^p}\le C(\|\nabla\dot{u}\|_{L^2}+\|\nabla u\|_{L^2}^2),
\ea\ee
\be\la{tb11}\ba
\|\nabla\dot{u}\|_{L^2}\le C(\|\div \dot{u}\|_{L^2}+\|\curl \dot{u}\|_{L^2}+\|\nabla u\|_{L^4}^{2}),
\ea\ee
for any $p\geq2$.
\end{lemma}
\begin{proof}
Set $u^{\perp}=(-u^2,\,u^1).$ Notice that $\dot{u}\cdot n=-\kappa|u|^{2}=\kappa|u|u^{\perp}\cdot n$ on $\partial\Omega$, then $(\dot{u}-\kappa|u|u^{\perp})\cdot n=0$ on $\partial\Omega$. As a result, Poincar\'{e}'s inequality still holds for $\dot{u}-\kappa|u|u^{\perp},$ that is,
$$\|\dot{u}-\kappa|u|u^{\perp}\|_{L^\frac{3}{2}}\le C\|\nabla(\dot{u}-\kappa|u|u^{\perp})\|_{L^\frac{3}{2}},$$
which leads to
\be\la{tb9}
\ba\|\dot{u}\|_{L^\frac{3}{2}}\le C(\|\nabla\dot{u}\|_{L^\frac{3}{2}}+\|\nabla u\|_{L^2}^{2}).
\ea\ee

On the other hand, by Sobolev's embedding theorem and \eqref{tb9}, one has
$$\|\dot{u}\|_{L^2}\leq C(\|\nabla\dot{u}\|_{L^\frac{3}{2}}+\|\dot{u}\|_{L^\frac{3}{2}})\leq C(\|\nabla\dot{u}\|_{L^2}+\|\nabla u\|_{L^2}^{2}),$$
$$\|\dot{u}\|_{L^p}\le C(\|\nabla\dot{u}\|_{L^2}+\|\dot{u}\|_{L^2})\leq C(\|\nabla\dot{u}\|_{L^2}+\|\nabla u\|_{L^2}^{2}).$$

Since $(\dot{u}-\kappa|u|u^{\perp})\cdot n=0$ on $\partial\Omega$, the inequality \eqref{tdu1} is still true if we replace $u$ by $\dot{u}-\kappa|u|u^{\perp}$. Consequently, it is easy to check that
\bnn
\|\nabla\dot{u}\|_{L^2}&\leq C(\|\div \dot{u}\|_{L^2}+\|\curl \dot{u}\|_{L^2}+\|\nabla u\|_{L^4}^2).
\enn
\end{proof}

In the following discussion $C$ will stand for a generic positive constant depending only on $\mu ,\,\,  \lambda , \,\,  \ga ,\,\,  a , \,\, \hat{\rho}, \,\,s, \,\, \Omega$ and $\|u_{0}\|_{H^s}$ and use $C(\alpha)$ to emphasize that $C$ depends on $\alpha.$  Let's first consider the following standard energy estimate for $(\rho,u)$ and preliminary $L^{2}$ bounds for $\nabla u$ .




\begin{lemma}\la{le2}
 Suppose $(\rho,u)$ is a smooth solution of
 \eqref{a1}--\eqref{ch1} on $\O \times (0,T]. $
  Then there exists a positive constant
  $C$ depending only  on $\mu,$  $\lambda$ and $\Omega$ such that
\be \la{a16} \sup_{0\le t\le T}\int
\left(\rho|u|^2+G(\rho)\right)dx + \int_0^{T}\int |\nabla u|^2  dxdt\le CC_0.\ee
\end{lemma}

\begin{proof}
It is easy to check that
\be\la{Pu1} \ba
 P_t+\div(Pu)+(\gamma-1)P\div u=0,
 \ea \ee
 or
 \be\la{Pu2} \ba
 P_t+\nabla P\cdot u+\gamma P\div u=0,
 \ea \ee
due to $(\ref{a1})_1$.
After a simple calculation, one has
\be\la{m0} \ba (G(\rho))_t + \div(G(\rho)u)+(P-P(\bar{\rho})) \div u=0.\ea \ee
Integrating over $\Omega$ and using the boundary condition \eqref{ch1}, we get
\be\la{m00} \ba
\left(\int G(\rho)dx\right)_t+\int(P-P(\bar{\rho})) \div udx=0.
\ea \ee

Note that $-\Delta u=-\nabla\div u+\nabla^\perp\curl u,$
we rewrite the equation of conservation of momentum $(\ref{a1})_2 $ as
\be\la{m1} \ba
\rho \dot{u} - (\lambda + 2\mu)\nabla\div u+\mu\nabla^{\perp}\curl u + \nabla P=0.
\ea \ee
Multiplying \eqref{m1} by $u$ and integrating over $\Omega$, together with \eqref{m00}, we obtain
\be\la{m8} \ba
&\left(\int(\rho |u|^{2}  +  G(\rho))dx\right)_t + (\lambda + 2\mu)\int(\div u)^{2}dx + \mu\int(\curl u)^{2}dx\\
&{-\mu\int_{\partial\Omega}2(\kappa-\vartheta)|u|^2ds\color{black}=0}.
\ea \ee
Notice that $\vartheta\geq \kappa$ on $\partial\Omega$, hence
\be\la{m9} \ba
&\sup_{0\le t\le T}\int
\left(\frac{1}{2}\n|u|^2+G(\rho)\right)dx + (\lambda + 2\mu)\int_0^{T}\int(\div u)^{2}dxdt\\
&\quad + \mu\int_0^{T}\int(\curl u)^{2}dxdt\\
&\leq C_0,
\ea \ee
along with \eqref{tdu1}, which yields \eqref{a16}.
\end{proof}
Define $$\sigma=\sigma(t)\triangleq\min\{1,t \},$$
 \be\la{As1}
  A_1(T) \triangleq \sup_{   0\le t\le T  }\left(\sigma\|\nabla u\|_{L^2}^2\right) + \int_0^{T} \sigma\int
 \n|\dot{u} |^2 dxdt,
  \ee
  and   \be \la{As2}
  A_2(T)  \triangleq\sup_{  0\le t\le T   }\sigma^2\int\n|\dot{u}|^2dx + \int_0^{T}\int
  \sigma^2|\nabla\dot{u}|^2dxdt.
\ee

Next, the main result in this section which guarantees the existence of a global classical solution of \eqref{a1}--\eqref{ch1} will be given. To obtain this proposition, we need to prove the following Lemmas \eqref{th00}--\eqref{le7} under the assumption \eqref{zz1}.
\begin{proposition}\la{pr1}  Under  the conditions of Theorem \ref{th1},
   there exists some  positive constant  $\ve$
    depending only on  $\mu$, $\lambda$, $a$, $\ga$, $\hat\rho,$ $s$, $\Omega$, and $M$  such that if
       $(\rho,u)$  is a smooth solution of
       \eqref{a1}--\eqref{ch1}  on $\Omega\times (0,T] $
        satisfying
 \be\la{zz1}
 \sup\limits_{
 \Omega\times [0,T]}\rho\le 2\hat{\rho},\quad
     A_1(T) + A_2(T) \le 2C_0^{1/3},
  \ee
 then the following estimates hold
        \be\la{zz2}
 \sup\limits_{\Omega\times [0,T]}\rho\le 7\hat{\rho}/4, \quad
     A_1(T) + A_2(T) \le  C_0^{1/3},
  \ee
   provided $C_0\le \ve.$
\end{proposition}
\begin{proof}Proposition \ref{pr1} is an obvious consequence of the Lemmas \ref{le5} and \ref{le7}.
\end{proof}
\begin{lemma}\la{nzc1} Suppose $(\n,u)$ is a smooth solution of \eqref{a1}-\eqref{ch1}   on $\O \times (0,T] $ satisfying \eqref{dt2}. Then there
exists a positive constant  $C $   depending only on $\mu,$  $\lambda,$ $\ga,$  $a,$ $\hat\rho,$  $s,$ $\Omega$ and $M$, such
that
\be\la{xuv1} \sup_{0\le t\le  \si(T) }\int \n |u|^{2+\nu}dx\le C(\hat\rho,M ) ,\ee
for some $\nu\in(0,\frac{1}{2})$.
\end{lemma}

\begin{proof}
Multiplying $(\ref{a1})_2$ by $(2+\nu)|u|^\nu u$, $\nu>0$ will be determined later, and integrating the resulting equation over $ \O$ lead  to
\bnn\ba & (2+\nu)\int|u|^{\nu}\rho \dot{u}\cdot u dx -(\lambda+2\mu)(2+\nu)\int|u|^{\nu}\nabla\div u\cdot u dx \\
& \quad + \mu(2 + \nu)\int|u|^{\nu}\nabla^{\perp}\curl u\cdot udx+(2+\nu)\int|u|^{\nu}u\cdot\nabla Pdx=0,\ea\enn
which, together with $(\ref{a1})_1$, gives
\be\la{xiugai1} \ba & \left(\int \n |u|^{2+\nu}dx\right)_t+(\lambda+2\mu)(2+\nu)\int\left(\div(|u|^{\frac{\nu}{2}}u)\right)^2dx\\
&\quad+ \mu(2+\nu)\int\left(\curl(|u|^{\frac{\nu}{2}}u)\right)^2dx- {\mu(2 + \nu)\int_{\partial\Omega}2(\kappa-\vartheta)|u|^{\nu+2}ds}\\
&\quad -(2+\nu)\int P\div(|u|^{\nu}u)dx \\
&=(\lambda+2\mu)(2+\nu)\int(\nabla|u|^{\frac{\nu}{2}}\cdot u)^2dx+\mu(2+\nu)\int(\nabla^{\perp}|u|^{\frac{\nu}{2}}\cdot u)^2dx.
\ea\ee
Notice that
$$|\nabla(|u|^{\frac{\nu}{2}+1})|\leq|\nabla(|u|^{\frac{\nu}{2}}u)|,\,\,|\nabla|u|^{\frac{\nu}{2}}\cdot u|\leq|\nabla|u|^{\frac{\nu}{2}}||u|=\frac{\nu}{2+\nu}|\nabla(|u|^{\frac{\nu}{2}+1})|$$
and
$$|\nabla^{\perp}|u|^{\frac{\nu}{2}}\cdot u|\leq\frac{\nu}{2+\nu}|\nabla(|u|^{\frac{\nu}{2}+1})|.$$
Therefore, by Young's inequality,
\be\la{xiugai2} \ba & \left(\int \n |u|^{2+\nu}dx\right)_t+(\lambda+2\mu)(2+\nu)\int\left(\div(|u|^{\frac{\nu}{2}}u)\right)^2dx \\
&\quad + \mu(2+\nu)\int\left(\curl(|u|^{\frac{\nu}{2}}u)\right)^2dx - {\mu(2 + \nu)\int_{\partial\Omega}2(\kappa-\vartheta)|u|^{\nu+2}ds}\\
&\leq(\lambda+2\mu)\nu^2\int|\nabla(|u|^{\frac{\nu}{2}}u)|^2dx+C\int\rho|u|^{2+\nu}dx+C\int\rho^{(2+\nu)\gamma-\frac{\nu}{2}}dx+C\int|\nabla u|^2dx\\
&\leq C(\lambda+2\mu)\nu^2\left(\int(\div(|u|^{\frac{\nu}{2}}u))^2dx+\int(\curl(|u|^{\frac{\nu}{2}}u))^2dx\right)+C\int\rho|u|^{2+\nu}dx\\
&\quad+C\int\rho^{(2+\nu)\gamma-\frac{\nu}{2}}dx+C\int|\nabla u|^2dx,
\ea\ee
where in the last inequality we have used the inequality
$$\int|\nabla(|u|^{\frac{\nu}{2}}u)|^2dx\leq C\left(\int(\div(|u|^{\frac{\nu}{2}}u))^2dx+\int(\curl(|u|^{\frac{\nu}{2}}u))^2dx\right),$$
since $u\cdot n=0$ on $\partial\Omega$.

Choosing $\nu=\min\{(\frac{\mu}{(\lambda+2\mu)C(\Omega)})^{\frac{1}{2}},\,\,\frac{1}{2}\},$ it follows from \eqref{xiugai2} that
$$\left(\int \rho |u|^{2+\nu}dx\right)_t\leq C\int\rho|u|^{2+\nu}dx+C\int\rho^{(2+\nu)\gamma-\frac{\nu}{2}}dx+C\int|\nabla u|^2dx.$$
By Gronwall's inequality and utilizing the following simple fact
\bnn\ba
& \left(\int\rho_0|u_0|^{2+\nu}dx\right)^{\frac{1}{2+\nu}}\leq C(\hat{\rho})\|\rho_0^{\frac{1}{2}}u_0\|_{L^{2}}^{\frac{2s-(1-s)\nu}{s}}\|u_0\|_{L^{\frac{2s}{1-s}}}^{\frac{\nu}{s}} \\
& \quad \leq C(\hat{\rho})C_0^{\frac{2s-(1-s)\nu}{s}}\|u_0\|_{{H}^{s}}^{\frac{\nu}{s}} ,
\ea\enn
one can derive \eqref{xuv1} immediately.
\end{proof}

Consider the problem
\bn\la{e480}\begin{cases}
{\rm div}v=f, \,\, &x\in\Omega, \\
v=0, \,\,&x\in{\partial\Omega},
\end{cases} \en
where $\Omega$ is a bounded domain in $R^{2}$ with $C^{\infty}$ boundary.

One can find the following conclusion in \cite[Theorem III.3.1]{GPG}.
\begin{lemma}\la{ll27}  There exists a linear operator $\mathcal{B} = [\mathcal{B}_1 , \mathcal{B}_2 ]$ enjoying
the properties:

1)The operator $$\mathcal{B}:\{f\in L^p(\O):\bar f =0\}\mapsto (W^{1,p}_0(\O))^2$$ is a bounded linear one, that is,
\be \|\mathcal{B}[f]\|_{W^{1,p}_0(\O)}\le C(p)\|f\|_{L^p(\O)}, \mbox{ for any }p\in (1,\infty).\ee

2) The function $v = \mathcal{B}[f]$ solve the problem \eqref{e480}.

3) if, moreover, $f$ can be written in the form $f = \div  g$ for a certain $g\in L^r(\O),$ $g\cdot n|_{\pa\O}=0,$  then
\be \|\mathcal{B}[f]\|_{L^{r}(\O)}\le C(r)\|g\|_{L^r(\O)}, \mbox{ for any }r  \in (1,\infty).\ee
\end{lemma}
\begin{lemma}   \la{th00}
Under the assumption \eqref{zz1}, one gets
\be\la{th00jl} \ba
\sup_{   0\le t\le T  }\sigma\|P-\bar{P}\|_{L^2}^2+\int_0^T\|P-\bar{P}\|_{L^2}^2dt\leq CC_0^{1/2},\,\,\int_0^T\sigma\|P-\bar{P}\|_{L^2}^2dt\leq CC_0^{3/4}.
\ea\ee
where $C=C(\hat{\rho})$ is a positive constant depending only on $\mu$,\,$\lambda$,\,$a$,\,$\gamma$,\,$\hat{\rho}$ and $\Omega$.
 \end{lemma}
\begin{proof}  Multiplying $\eqref{a1}_2$ by $\mathcal{B}[P-\Bar P]$ and integrating over $\Omega,$ we obtain
\be\la{e4} \ba &
\int(P-\Bar P)^2 dx \\&= \left(\int\rho u\cdot\mathcal{B}[P-\Bar P] dx\right)_t-\int\rho u\cdot\nabla\mathcal{B}[P-\Bar P]\cdot udx - \int\rho u\cdot\mathcal{B}[P_t-\Bar P_t]  dx \\
& \quad  +\mu\int\nabla u\cdot\nabla\mathcal{B}[P-\Bar P] dx + (\lambda+\mu)\int(P-\Bar P)\div udx \\
& \leq \left(\int\rho u\cdot\mathcal{B}[P-\Bar P] dx\right)_t+C\| u\|_{L^4}^{2}\|P-\Bar P\|_{L^{2}} +C\|  u\|_{L^2}\|\nabla u\|_{L^2}\\
& \quad  +C\|P-\Bar P\|_{L^2}\|\nabla u\|_{L^2}\\
& \leq \left(\int\rho u\cdot\mathcal{B}[P-\Bar P] dx\right)_t+\de \|P-\Bar P\|_{L^2}^2 +C(\de)\|\na u\|_{L^2}^2,
\ea\ee where in the first inequality we have used \be\ba \|\mathcal{B}[P_t-\bar P_t]\|_{L^2}&=\|\mathcal{B} [\div(Pu)] + (\ga-1) \mathcal{B} [P\div u-\ol{P\div u}]\|_{L^2} \\&\le C\|\na u\|_{L^2}.\ea\ee
Combining \eqref{e4}, \eqref{a16} and Lemma \ref{ll27} gives
\be\la{xz1} \ba
\int_0^T\|P-\bar{P}\|_{L^2}^2dt\leq CC_0^{1/2}.
\ea\ee

By $\eqref{a1}_1$, we have
\be\la{xz2} \ba
(P-\bar{P})_t+u\cdot\nabla(P-\bar{P})+\gamma P\div u-(\gamma-1)\overline{(P-\bar{P})\div u}=0.
\ea\ee
Multiplying the above equation by $2\sigma(P-\bar{P})$ and integrating with respect to $x$ yields
\be\la{xz3} \ba
\left(\sigma\int(P-\bar{P})^2dx\right)_t\leq C(\sigma+\sigma')\int(P-\bar{P})^2dx+C\sigma\int|\nabla u|^2dx,
\ea\ee
which together with \eqref{a16} and \eqref{xz1} leads to
\be\la{xz4} \ba
\sup_{   0\le t\le T  }\sigma\|P-\bar{P}\|_{L^2}^2\leq CC_0^{1/2}.
\ea\ee
\eqref{a16}, \eqref{e4}, \eqref{xz1} and \eqref{xz4} imply
\be\la{xz5} \ba
&\int_0^T\sigma\|P-\bar{P}\|_{L^2}^2dt\\
&\leq\int_0^T\left(\sigma\int\rho u\cdot\mathcal{B}[P-\Bar P] dx\right)_tdt-\int_0^T\sigma'\int\rho u\cdot\mathcal{B}[P-\Bar P] dxdt+C(\de)\int_0^T\|\na u\|_{L^2}^2dt\\
&\leq C\sup_{   0\le t\le T  }\left(\int\rho|u|^2dx\right)^{1/2}\sup_{   0\le t\le T  }\left(\sigma\|P-\bar{P}\|_{L^2}^2\right)^{1/2}+C(\hat{\rho})C_0^{1/2}+CC_0\\
&\leq CC_0^{3/4}.
\ea\ee
\end{proof}
\begin{lemma}\la{xcrle1}
 Suppose $(\n,u)$ is a smooth solution of
 \eqref{a1}-\eqref{ch1} satisfying \eqref{zz1} on $\O \times (0,T].$
  Then there is a positive constant
  $C$ depending on $\mu,\,\,\lambda,\,\,\gamma,\,\,a,\,\,\hat{\rho},$  and $\Omega$ such that
  \be\la{h14}
  A_1(T) \le  C C_0^{1/2} + C\int_0^{T}\int\sigma|\nabla u|^3dx dt,
  \ee
 and
  \be\la{h15}
    A_2(T)
    \le   C C_0^{2/3} + CA_1(T)  + C\int_0^{T}\int \sigma^2 |\nabla u|^4 dxdt.
   \ee
\end{lemma}
\begin{proof}
 The proof of this lemma is due to Hoff \cite{H3}. For $m\ge 0,$ multiplying $(\ref{a1})_2 $ by
$\sigma^m \dot{u},$   and then integrating the resulting equality over
$\Omega$ lead  to
\be\la{I0} \ba  \int \sigma^m \rho|\dot{u}|^2dx&= -\int\sigma^m \dot{u}\cdot\nabla Pdx + (\lambda+2\mu)\int\sigma^m \nabla\div u\cdot\dot{u}dx \\
 &\quad- \mu\int\sigma^m \nabla^{\perp}\curl u\cdot\dot{u}dx \\
& \triangleq \sum_{i=1}^{3}I_i. \ea \ee

By \eqref{Pu1} and $(\ref{a1})_1$,
 \be\la{I1} \ba
I_1 = & - \int \sigma^m \dot{u}\cdot\nabla Pdx \\
= & \left(\int\sigma^m(P-\bar{P})\div u dx\right)_{t} - m\sigma^{m-1}\sigma'\int (P-\bar{P})\div udx - \int \sigma^{m}\div u P_{t}dx \\
&- \int\sigma^{m}u\cdot\nabla u\cdot\nabla Pdx\\
= & \left(\int\sigma^m(P-\bar{P})\div u dx\right)_{t} - m\sigma^{m-1}\sigma'\int (P-\bar{P})\div udx + \int \sigma^{m}\div u\div(Pu)dx \\
&+ (\gamma-1)\int\sigma^{m}P(\div u)^{2}dx  - \int_{\partial\Omega}\sigma^{m}Pu\cdot\nabla u\cdot nds + \int\sigma^{m}P\div(u\cdot\nabla u)dx\\
= & \left(\int\sigma^mP\div u dx\right)_{t} - m\sigma^{m-1}\sigma'\int (P-\bar{P})\div udx + \int\sigma^{m}P\nabla u:\nabla udx \\
&+ (\gamma-1)\int\sigma^{m}P(\div u)^{2}dx - \int_{\partial\Omega}\sigma^{m}Pu\cdot\nabla u\cdot nds\\
 \leq &\left(\int\sigma^m(P-\bar{P})\div u dx\right)_{t} + C\|\nabla u\|_{L^{2}}^{2} + Cm\sigma^{m-1}\sigma'\|\nabla u\|_{L^{2}},
\ea \ee
where in the last inequality we have used \eqref{a16} and the following fact
\bnn \ba
- \int_{\partial\Omega}\sigma^{m}Pu\cdot\nabla u\cdot nds&=\int_{\partial\Omega}\sigma^{m}\kappa  P|u|^{2}ds \\
&\leq C\int_{\partial\Omega}\sigma^{m}|u|^{2}ds \\
& \leq C\sigma^{m}\|\nabla u\|_{L^{2}}^{2}.
\ea  \enn
A direct calculation leads to
\bnn \ba
I_2 & =  (\lambda+2\mu)\int\sigma^m \nabla\div u\cdot\dot{u}dx \\
& = (\lambda+2\mu)\int_{\partial\Omega}\sigma^m\div u(\dot{u}\cdot n)ds - (\lambda+2\mu)\int\sigma^m\div u\div \dot{u}dx  \\
& = (\lambda+2\mu)\int_{\partial\Omega}\sigma^m\div u(u\cdot\nabla u\cdot n)ds - \frac{\lambda+2\mu}{2}\left(\int\sigma^{m}(\div u)^{2}dx\right)_{t} \\
&\quad - (\lambda+2\mu)\int\sigma^m\div u\div(u\cdot\nabla u)dx + \frac{m(\lambda+2\mu)}{2}\sigma^{m-1}\sigma'\int(\div u)^{2}dx \\
& = (\lambda+2\mu)\int_{\partial\Omega}\sigma^m\div u(u\cdot\nabla u\cdot n)ds - \frac{\lambda+2\mu}{2}\left(\int\sigma^{m}(\div u)^{2}dx\right)_{t} \\
&\quad +\frac{\lambda+2\mu}{2}\int\sigma^{m}(\div u)^{3}dx- (\lambda+2\mu)\int\sigma^m\div u\nabla u:\nabla udx  \\
&\quad + \frac{m(\lambda+2\mu)}{2}\sigma^{m-1}\sigma'\int(\div u)^{2}dx\\
& \leq - \frac{\lambda+2\mu}{2}\left(\int\sigma^{m}(\div u)^{2}dx\right)_t+Cm\sigma^{m-1}\sigma'\|\nabla u\|_{L^{2}}^{2}+C\int\sigma^{m}|\nabla u|^{3}dx\\
&\quad +\frac{1}{4}\sigma^{m}\|\rho^{\frac{1}{2}}\dot{u}\|_{L^{2}}^{2}+C\sigma^{m}\|\nabla u\|_{L^{2}}^{4}+C(\hat{\rho})\|\nabla u\|_{L^{2}}^{2}.
\ea\enn
On the other hand, by \eqref{h19} and \eqref{bdd2},
\bnn \ba
&\left|\int_{\partial\Omega}\div u(u\cdot\nabla u\cdot n)ds\right| \\
&=\left|\int_{\partial\Omega}\kappa \div u|u|^{2}ds\right|\\
& \leq C\left|\int_{\partial\Omega} F|u|^{2}ds\right|+C\left|\int_{\partial\Omega}(P-\bar{P})|u|^{2}ds\right|\\
& \leq C(\|\nabla F\|_{L^{2}}\|u\|_{L^{4}}^{2}+\|F\|_{L^{2}}\|u\|_{L^{6}}^{2}+\|F\|_{L^{6}}\|u\|_{L^{3}}\|\nabla u\|_{L^{2}})+C(\hat{\rho})\|\nabla u\|_{L^{2}}^{2}\\
& \leq\frac{1}{4}\|\rho^{\frac{1}{2}}\dot{u}\|_{L^{2}}^{2}+C\|\nabla u\|_{L^{2}}^{4}+C(\hat{\rho})\|\nabla u\|_{L^{2}}^{2}.
\ea  \enn
Hence,
\be\la{I21} \ba
I_2 & \leq - \frac{\lambda+2\mu}{2}\left(\int\sigma^{m}(\div u)^{2}dx\right)_t+Cm\sigma^{m-1}\sigma'\|\nabla u\|_{L^{2}}^{2}+C\int\sigma^{m}|\nabla u|^{3}dx\\
&\quad +\frac{1}{4}\sigma^{m}\|\rho^{\frac{1}{2}}\dot{u}\|_{L^{2}}^{2}+C\sigma^{m}\|\nabla u\|_{L^{2}}^{4}+C(\hat{\rho})\|\nabla u\|_{L^{2}}^{2}.
\ea\ee
Similarly,
\be\la{I3}\ba
I_3 & = -\mu\int\sigma^{m}\nabla^{\perp}\curl u\cdot\dot{u}dx \\
& = \mu\int_{\partial\Omega}\sigma^{m}\curl u(\dot{u}\cdot\omega)ds - \mu\int\sigma^{m}\curl u\curl\dot{u}dx \\
& = \mu\int_{\partial\Omega}2\sigma^{m}(\kappa-\vartheta) u\cdot\omega(\dot{u}\cdot\omega)ds\color{black}-\frac{\mu}{2}\left(\int\sigma^{m}(\curl u)^{2}dx\right)_t \\& \quad+ \frac{\mu m}{2}\sigma^{m-1}\sigma'\int(\curl u)^{2}dx- \mu\int\sigma^{m}\curl u\curl(u\cdot\nabla u)dx \\
& = \mu\left(\int_{\partial\Omega}\sigma^{m}(\kappa-\vartheta) |u|^2ds\right)_t\color{black}-\frac{\mu}{2}\left(\int\sigma^{m}(\curl u)^{2}dx\right)_t + \frac{\mu m}{2}\sigma^{m-1}\sigma'\int(\curl u)^{2}dx \\
& \quad - \mu m\sigma^{m-1}\sigma'\int_{\partial\Omega}(\kappa-\vartheta)|u|^{2}ds \color{black}- \mu\int\sigma^{m}\curl u(\div u\curl u + u\cdot\nabla\curl u)dx\\
& \quad+\mu\int_{\partial\Omega}2\sigma^{m}(\kappa-\vartheta) u\cdot\omega(u\cdot\nabla u\cdot\omega)ds\\
& = -\frac{\mu}{2}\left(\int\sigma^{m}(\curl u)^{2}dx\right)_t + \frac{\mu m}{2}\sigma^{m-1}\sigma'\int(\curl u)^{2}dx - \mu\int\sigma^{m}(\curl u)^{2}\div udx \\
& \quad - \frac{\mu}{2}\int\sigma^{m}\nabla(\curl u)^{2}\cdot udx+\mu\left(\int_{\partial\Omega}\sigma^m(\kappa-\vartheta) |u|^2ds\right)_t\\
&\quad - \mu m\sigma^{m-1}\sigma'\int_{\partial\Omega}(\kappa-\vartheta)|u|^{2}ds -\mu\int_\Omega 2\sigma^{m}\nabla^\perp((\kappa-\vartheta) u\cdot\omega)\cdot(u\cdot\nabla u)dx\\
&\quad+\mu\int_\Omega 2\sigma^{m}(\kappa-\vartheta) u\cdot\omega \curl (u\cdot\nabla u)dx\\
& = -\frac{\mu}{2}\left(\int\sigma^{m}(\curl u)^{2}dx\right)_t + \frac{\mu m}{2}\sigma^{m-1}\sigma'\int(\curl u)^{2}dx - \mu\int\sigma^{m}(\curl u)^{2}\div udx  \\
& \quad + \frac{\mu}{2}\int\sigma^{m}(\curl u)^{2}\div udx+\mu\left(\int_{\partial\Omega}\sigma^m(\kappa-\vartheta) |u|^2ds\right)_t\\
&\quad- \mu m\sigma^{m-1}\sigma'\int_{\partial\Omega}(\kappa-\vartheta)|u|^{2}ds -\mu\int_\Omega 2\sigma^{m}\nabla^\perp((\kappa-\vartheta) u\cdot\omega)\cdot(u\cdot\nabla u)dx\\
&\quad+\mu\int_\Omega 2\sigma^{m}(\kappa-\vartheta) u\cdot\omega \,\curl u \,\div udx +\mu\int_\Omega 2\sigma^{m}(\kappa-\vartheta) u\cdot\omega \,u\cdot\nabla\curl u dx \\
& \leq -\frac{\mu}{2}\left(\int\sigma^{m}(\curl u)^{2}dx\right)_t+\mu\left(\int_{\partial\Omega}\sigma^m(\kappa-\vartheta) |u|^2ds\right)_t \color{black}+ Cm\sigma^{m-1}\sigma'\|\nabla u\|_{L^{2}}^{2}\\
&\quad + C\sigma^{m}\|\nabla u\|_{L^{3}}^{3}+C\sigma^{m}\|\kappa-\vartheta\|_{H^1}\| u\|_{L^{6}}^{2}\|\nabla u\|_{L^3}+C\sigma^{m}\|\kappa-\vartheta\|_{L^6}\| u\|_{L^{6}}\|\nabla u\|_{L^3}^2\\
&\quad +C\sigma^{m}\|\kappa-\vartheta\|_{L^6}\| u\|_{L^{6}}^2\|\nabla \curl u\|_{L^2}\\
& \leq -\frac{\mu}{2}\left(\int\sigma^{m}(\curl u)^{2}dx\right)_t+\mu\left(\int_{\partial\Omega}\sigma^m(\kappa-\vartheta) |u|^2ds\right)_t \color{black}+ Cm\sigma^{m-1}\sigma'\|\nabla u\|_{L^{2}}^{2}\\
&\quad + C\sigma^{m}\|\nabla u\|_{L^{3}}^{3}+C\sigma^{m}(\|\nabla u\|_{L^{2}}^{2}+\|\nabla u\|_{L^{2}}^{4})+\frac{1}{4}\sigma^{m}\|\rho^{\frac{1}{2}}\dot{u}\|_{L^{2}}^{2}.
\ea \ee
By virtue of \eqref{I1}-\eqref{I3}, it follows from $(\ref{I0})$ that
\be\la{I4}\ba
&\left((\lambda+2\mu)\int\sigma^{m}(\div u)^{2}dx+\mu\int\sigma^{m}(\curl u)^{2}dx\right)_{t}+\int\sigma^{m}\rho|\dot{u}|^{2}dx\\
&\quad-2\left(\int\sigma^{m}(P-\bar{P})\div udx+\mu\int_{\partial\Omega}\sigma^m(\kappa-\vartheta) |u|^2ds\color{black}\right)_{t} \\
& \leq Cm\sigma^{m-1}\sigma'\|\nabla u\|_{L^{2}}+C\sigma^{m}\|\nabla u\|_{L^{2}}^{4}+C\|\nabla u\|_{L^{2}}^{2}+C\sigma^{m}\|\nabla u\|_{L^{3}}^{3}.
\ea \ee
Integrating over $(0,T]$, one has, for $m\geq1$,
\be\la{I40}\ba
&(\lambda+2\mu)\int\sigma^{m}(\div u)^{2}dx+\mu\int\sigma^{m}(\curl u)^{2}dx+\int_0^T\int\sigma^{m}\rho|\dot{u}|^{2}dxdt\\
&\quad-2\int\sigma^{m}(P-\bar{P})\div udx-2\mu\int_{\partial\Omega}\sigma^m(\kappa-\vartheta) |u|^2ds \\
& \leq CC_{0}^{1/2}+C\int_0^T\sigma^{m}\|\nabla u\|_{L^{2}}^{4}dt+C\int_0^T\sigma^{m}\|\nabla u\|_{L^{3}}^{3}dt.
\ea \ee
By \eqref{tdu1}, \eqref{a16} and \eqref{th00jl},
\be\la{I5}\ba
&\sigma^{m}\|\nabla u\|_{L^{2}}^{2}+\int_0^T\int\sigma^{m}\rho|\dot{u}|^{2}dxdt \\
& \leq CC_{0}^{1/2}+C\int_0^T\sigma^{m}\int|\nabla u|^{3}dxdt+C\int_0^T\sigma^{m}\|\nabla u\|_{L^{2}}^{4}dt.
\ea \ee
Choosing $m=1$, together with  \eqref{a16} and the assumption \eqref{zz1}, we establish $\eqref{h14}$.

Now we will prove \eqref{h15}. $(\ref{a1})_2 $ can be rewritten as
\be\la{xdy1}\ba
\rho\dot{u}=\nabla F - \mu\nabla^{\perp}\curl u.
\ea \ee
Operating $ \sigma^{m}\dot{u}^{j}[\pa/\pa t+\div
(u\cdot)] $ to $ (\ref{xdy1})^j ,$  summing all the equalities with respect to $j$, then integrating over $\Omega,$ by $(\ref{a1})_1 $ and \eqref{ch1}, we get
\be\la{ax1}\ba &\left(\frac{\sigma^{m}}{2}\int\rho|\dot{u}|^{2}dx\right)_t-\frac{m}{2}\sigma^{m-1}\sigma'\int\rho|\dot{u}|^{2}dx \\
& = \int(\sigma^{m}\dot{u}\cdot\nabla F_t+\dot{u}^{j}\div(u\partial_jF))dx-\mu\int\sigma^{m}(\dot{u}\cdot\nabla^{\perp}\curl u_t+\dot{u}^{j}\div((\nabla^{\perp}\curl u)^ju))dx \\
& \triangleq J_1+J_2.
\ea\ee
Next, we will estimate $J_1$ and $J_2$.

By \eqref{Pu2}, a direct computation yields
\be\la{ax2}\ba J_1 & =\int\sigma^{m}\dot{u}\cdot\nabla F_tdx+\int\sigma^{m}\dot{u}^{j}\div(u\partial_jF)dx \\
& = \int_{\partial\Omega}\sigma^{m}F_t\dot{u}\cdot nds-\int\sigma^{m}F_t\,\div\dot{u}dx+ \int\sigma^{m}\dot{u}\cdot\nabla\div(uF)dx\\
&\quad-\int\sigma^{m}\dot{u}^{j}\div(\nabla_ju\,F)dx \\
& = \int_{\partial\Omega}\sigma^{m}F_t\dot{u}\cdot nds - (\lambda+2\mu)\int\sigma^{m}((\div\dot{u})^{2}-\div\dot{u}\,\div(u\cdot\nabla u))dx \\
& \quad + \int\sigma^{m}\div\dot{u}\,P_tdx + \int_{\partial\Omega}\sigma^{m}\div(uF)\,\dot{u}\cdot nds- \int\sigma^{m}\div(uF)\,\div\dot{u}dx \\
& \quad  - \int_{\partial\Omega}\sigma^{m}F\dot{u}^{j}\nabla_ju\cdot nds+\int\sigma^{m}F\nabla\dot{u}:\nabla udx-(\gamma-1)\overline{P\div u}\int\sigma^{m}\div \dot{u} dx\\
& = \int_{\partial\Omega}\sigma^{m}F_t\dot{u}\cdot nds + \int_{\partial\Omega}\sigma^{m}\div(F\,u)\,\dot{u}\cdot nds -\int_{\partial\Omega}\sigma^{m}F\dot{u}\cdot\nabla u\cdot nds \\
& \quad - (\lambda+2\mu)\int\sigma^{m}(\div\dot{u})^{2}dx + (\lambda+2\mu)\int\sigma^{m}\div\dot{u}\,\nabla u:\nabla udx \\
& \quad -\gamma\int\sigma^{m} P\div\dot{u}\,\div udx-\int\sigma^{m}F\div \dot{u}\,\div udx+\int\sigma^{m}F\nabla\dot{u}:\nabla u dx\\
&\quad-(\gamma-1)\overline{P\div u}\int\sigma^{m}\div \dot{u} dx,
\ea\ee
 where in the third   equality we have used \eqref{Pu1} and the fact
\bnn\ba F_t&=(2\mu+\lm)\div u_t-P_t+\bar P_t\\&=(2\mu+\lm)\div\dot  u-(2\mu+\lm)\div(u\cdot\na u) +u\cdot\na P+\ga P\div u-(\gamma-1)\overline{P\div u}\\&=(2\mu+\lm)\div\dot  u-(2\mu+\lm)\na u:\na u  - u\cdot\na F+\ga P\div u-(\gamma-1)\overline{P\div u}.\ea\enn

Now we have to estimate the three boundary terms on the righthand side of \eqref{ax2}. In the following discussion, we will take advantage of Lemmas \ref{le3}, \ref{uup1}, Poincar\'{e}'s and Young's inequalities.

For the first boundary term,
\be\la{ax3}\ba &\int_{\partial\Omega}\sigma^{m}F_t\dot{u}\cdot nds \\
& = \left(\int_{\partial\Omega}\sigma^{m}Fu\cdot\nabla u\cdot nds\right)_t-m\sigma^{m-1}\sigma'\int_{\partial\Omega}Fu\cdot\nabla u\cdot nds\\
&\quad-\int_{\partial\Omega}\sigma^{m}F(u\cdot\nabla u\cdot n)_tds \\
& = -\left(\int_{\partial\Omega}\sigma^{m}\kappa  F|u|^{2}ds\right)_t+m\sigma^{m-1}\sigma'\int\curl((u\cdot\omega)\kappa  Fu)dx\\
&\quad+2\int\sigma^{m}\curl(\kappa  F (u\cdot\omega)u_t)dx \\
& = -\left(\int_{\partial\Omega}\sigma^{m}\kappa F|u|^{2}ds\right)_t+m\sigma^{m-1}\sigma'\int\curl((u\cdot\omega)\kappa  Fu)dx\\ &\quad+2\int\sigma^{m}\curl((u\cdot\omega)\kappa  F\dot{u})dx + 2\int\sigma^{m}\nabla_j^{\perp}((u\cdot\omega)\kappa  F)(u\cdot\nabla u)^{j}dx\\
&\quad+2\int\sigma^{m}\curl u\nabla((u\cdot\omega)\kappa  F)\cdot udx \\
&\leq-\left(\int_{\partial\Omega}\sigma^{m}\kappa F|u|^{2}ds\right)_t+Cm\sigma^{m-1}(\|F\|_{L^2}+\|\nabla F\|_{L^2})\|u\|_{L^4}^2 \\
&\quad+Cm\sigma^{m-1}\|F\|_{L^6}\|u\|_{L^3}\|\nabla u\|_{L^2}+C\sigma^m(\|F\|_{L^2}+\|\nabla F\|_{L^2})\|u\|_{L^3}\|\dot{u}\|_{L^6}\\
&\quad+C\sigma^m(\|F\|_{L^6}\|\nabla u\|_{L^2}\|\dot{u}\|_{L^3}+\|F\|_{L^6}\|u\|_{L^6}^2\|\nabla u\|_{L^2}+\|F\|_{L^6}\|u\|_{L^3}\|\nabla \dot{u}\|_{L^2})\\
&\quad+C\sigma^m(\|\nabla F\|_{L^2}\|u\|_{L^8}^2\|\nabla u\|_{L^4}+\|\nabla u\|_{L^4}^{2}\|F\|_{L^6}\|u\|_{L^3})\\
&\leq-\left(\int_{\partial\Omega}\sigma^{m}\kappa F|u|^{2}ds\right)_t+\frac{\delta}{4}\sigma^m\|\nabla\dot{u}\|_{L^2}^2+C\sigma^{m}\|\rho^{\frac{1}{2}}\dot{u}\|_{L^2}^2\|\nabla u\|_{L^2}^2+C\sigma^m\|\rho\dot{u}\|_{L^2}^2\\
&\quad+C\sigma^m(\|\nabla u\|_{L^4}^4+\|\nabla u\|_{L^2}^2+\|\nabla u\|_{L^2}^4+\|P-\bar{P}\|_{L^2}^2\|\nabla u\|_{L^2}^2+\|P-\bar{P}\|_{L^2}^2)\\
&\quad+Cm\sigma^{m-1}(\|\nabla u\|_{L^2}^2+\|\nabla u\|_{L^2}^4+\|\rho\dot{u}\|_{L^2}^2).
\ea\ee
For the second boundary term,
\be\la{cax31}\ba &\int_{\partial\Omega}\sigma^{m}\div(F\,u)\,\dot{u}\cdot nds \\
&=\int_{\partial\Omega}\sigma^{m}(\nabla F\cdot u)(\dot{u}\cdot n)ds+\frac{1}{\lambda+2\mu}\int_{\partial\Omega}(F^2+PF-\bar P F)\dot{u}\cdot n ds  \\
&=\int\sigma^{m}\nabla_j^{\perp}(\kappa |u|^{2}(u\cdot\omega))\partial_jF dx-\frac{1}{\lambda+2\mu}\int_{\partial\Omega}\kappa(F^2+PF-\bar P F)|u|^2 ds\\
&\le C\sigma^{m}(\|u\|_{L^6}^3\|\nabla F\|_{L^2}+\|\nabla u\|_{L^4}\|u\|_{L^8}^2\|\nabla F\|_{L^2})+ C\sigma^{m}\|\nabla F\|_{L^2}\|F\|_{L^6}\|u\|_{L^6}^2\\
&\quad+C\sigma^{m}(\|F\|_{L^6}\|\nabla u\|_{L^2}\|u\|_{L^3}+\|F\|_{L^6}^2\|\nabla u\|_{L^2}\|u\|_{L^6}+\|\nabla F\|_{L^2}\|u\|_{L^4}^2)  \\
&\leq C\sigma^m(\|\nabla u\|_{L^2}^3\|\rho\dot{u}\|_{L^2}+\|\nabla u\|_{L^4}\|\nabla u\|_{L^2}^2\|\rho\dot{u}\|_{L^2}+\|\nabla u\|_{L^2}^4+\|\nabla u\|_{L^4}\|\nabla u\|_{L^2}^3)\\
&\quad+C\sigma^{m}(\|\rho\dot{u}\|_{L^2}+\|\nabla u\|_{L^2})\|\rho\dot{u}\|_{L^2}^{2/3}(\|\nabla u\|_{L^2}+\|P-\bar{P}\|_{L^2})^{1/3}\|\nabla u\|_{L^2}^2  \\
&\quad + C\sigma^{m}(\|\rho\dot{u}\|_{L^2}+\|\nabla u\|_{L^2}+\|P-\bar{P}\|_{L^2})\|\nabla u\|_{L^2}^2\\
&\quad + C\sigma^{m}(\|\rho\dot{u}\|_{L^2}^2+\|\nabla u\|_{L^2}^2+\|P-\bar{P}\|_{L^2}^2)\|\nabla u\|_{L^2}^2\\
&\le C\sigma^m(\|\nabla u\|_{L^2}^2\|\rho^{\frac{1}{2}}\dot{u}\|_{L^2}^2+\|\nabla u\|_{L^4}^4+\|\nabla u\|_{L^2}^2+\|\nabla u\|_{L^2}^4),
\ea\ee
where we have utilized the following fact
\bnn\ba
\int_{\partial\Omega}\sigma^{m}(\nabla F\cdot u)(\dot{u}\cdot n)ds
&=-\int_{\partial\Omega}\sigma^{m}\kappa \nabla F\cdot u|u|^{2}ds  \\
&=-\int_{\partial\Omega}\sigma^{m}\kappa |u|^{2}(u\cdot\omega)\nabla F\cdot\omega ds\\
&=-\int\sigma^{m}\curl(\kappa |u|^{2}(u\cdot\omega)\nabla F) dx\\
&=\int\sigma^{m}\nabla_j^{\perp}(\kappa |u|^{2}(u\cdot\omega))\partial_jF dx.\\
\ea\enn
Since $\dot{u}\cdot n=-u\cdot\nabla n\cdot u=-\kappa|u|^2$ on $\partial \Omega$, one can decompose the vector $\dot{u}$ as $\dot{u}=-\kappa|u|^2\,n+(\dot{u}\cdot\omega)\,\omega$. Hence, for the third boundary term,
\be\la{cax32}\ba
&\int_{\partial\Omega}\sigma^{m} F\dot{u}\cdot\nabla u\cdot n ds \\
&=-\int_{\partial\Omega}\sigma^{m} \kappa|u|^2\,F\,n\cdot\nabla u\cdot nds+\int_{\partial\Omega}\sigma^{m}(\dot{u}\cdot\omega)F\,\omega\cdot\nabla u\cdot nds \\
&=-\int\sigma^{m}\div(\kappa|u|^2\,F\, n\cdot\nabla u)dx-\int_{\partial\Omega}\sigma^{m}(\dot{u}\cdot\omega)F\,\omega\cdot\nabla n\cdot uds \\
&=-\int\sigma^{m}\nabla(\kappa|u|^2\,F \,n^{i})\cdot\nabla_i udx-\int\sigma^{m}(u\cdot\nabla n\cdot u)\,F\,n^i\,\partial_i\div udx \\
&\quad-\int_{\partial\Omega}\sigma^{m}(\dot{u}\cdot\omega)(u\cdot\omega)F\,\omega\cdot\nabla n\cdot \omega ds \\
&=-\int\sigma^{m}\nabla(\kappa|u|^2\,F\, n^{i})\cdot\nabla_i udx-\frac{1}{2(\lambda+2\mu)}\int\sigma^{m} \kappa|u|^2\,n\cdot\nabla F^2dx \\
&\quad-\frac{1}{\lambda+2\mu}\int\sigma^{m} \kappa|u|^2\,F\,n\cdot\nabla Pdx-\int_{\partial\Omega}\sigma^{m}(\dot{u}\cdot\omega)(u\cdot\omega)\kappa F ds \\
&=-\int\sigma^{m}\nabla(\kappa|u|^2 \,F\, n^{i})\cdot\nabla_i udx-\frac{1}{2(\lambda+2\mu)}\int_{\partial\Omega}\sigma^{m} F^{2}\,\kappa|u|^2ds \\
&\quad+\frac{1}{2(\lambda+2\mu)}\int\sigma^{m} F^{2}\div(\kappa|u|^2\,n)dx-\frac{1}{\lambda+2\mu}\int_{\partial\Omega}\sigma^{m}PF\,\kappa|u|^2ds \\
&\quad+\frac{1}{\lambda+2\mu}\int\sigma^{m} P\div(\kappa|u|^2F\, n)dx-\int_{\partial\Omega}\sigma^{m}(\dot{u}\cdot\omega)(u\cdot\omega)\kappa Fds \\
&\le\frac{\delta}{8}\sigma^{m}\|\nabla\dot{u}\|_{L^2}^{2}+C\sigma^{m}\|\rho^{\frac{1}{2}}\dot{u}\|_{L^2}^2\|\nabla u\|_{L^2}^2+C\sigma^{m}(\|\nabla u\|_{L^2}^{2}+\|\nabla u\|_{L^2}^{4}+\|\nabla u\|_{L^4}^{4}).
\ea\ee
Together with the estimates of all boundary terms, we deduce from \eqref{ax2} that
\be\la{ycax1}\ba
J_1&\leq\frac{\delta}{2}\sigma^{m}\|\nabla\dot{u}\|_{L^2}^2+C\sigma^{m}\|\rho^{\frac{1}{2}}\dot{u}\|_{L^2}^2\|\nabla u\|_{L^2}^2+C\sigma^{m}\|\rho\dot{u}\|_{L^2}^2+ C\sigma^{m}\|\nabla u\|_{L^2}^2\\
&\quad + C\sigma^{m}\|\nabla u\|_{L^4}^4-(\lambda+2\mu)\sigma^m\|\div \dot{u}\|_{L^2}^2-\left(\int_{\partial\Omega}\sigma^{m}\kappa  F|u|^2ds\right)_t\\
&\quad + C\sigma^{m}(\|\nabla u\|_{L^2}^4+\|P-\bar{P}\|_{L^2}^2\|\nabla u\|_{L^2}^2+\|P-\bar{P}\|_{L^2}^2)\\
&\quad + Cm\sigma^{m-1}(\|\nabla u\|_{L^2}^2+\|\nabla u\|_{L^2}^4+\|\rho\dot{u}\|_{L^2}^2).
\ea\ee
For $J_2$, integrating by part shows that
\be\la{ax4XG}\ba J_2&=-\mu\int\sigma^{m}\dot{u}\cdot\nabla^{\perp}\curl u_tdx-\mu\int\sigma^{m}\dot{u}^{j}\div(u\nabla_j^{\perp}\curl u)dx \\
& = \mu\int_{\partial\Omega}\sigma^{m}\curl u_t\dot{u}\cdot\omega ds \color{black}- \mu\int\sigma^{m}\curl\dot{u}\curl u_tdx \\
& \quad - \mu\int\sigma^{m}\dot{u}^{j}(\nabla_j^{\perp}\div(u\curl u)-\div(\nabla_j^{\perp}u\curl u))dx \\
& =  \mu\int_{\partial\Omega}\sigma^{m}\curl u_t\dot{u}\cdot\omega ds \color{black}-\mu\int\sigma^{m}(\curl\dot{u})^{2}dx+\mu\int\sigma^{m}\curl\dot{u}\curl(u\cdot\nabla u)dx \\
& \quad -\mu\int\sigma^{m}\dot{u}\cdot\nabla^{\perp}\div(u\curl u)dx+\mu\int\sigma^{m}\dot{u}^{j}\div(\nabla_j^{\perp}u\curl u)dx \\
& =  \mu\int_{\partial\Omega}\sigma^{m}\curl u_t\dot{u}\cdot\omega ds \color{black}-\mu\int\sigma^{m}(\curl\dot{u})^{2}dx+\mu\int\sigma^{m}\curl\dot{u}\,\div(u\curl u)dx \\
& \quad +\mu\int_{\partial\Omega}\sigma^{m}\div(u\curl u)\dot{u}\cdot\omega ds-\mu\int\sigma^{m}\curl\dot{u}\,\div(u\curl u)dx \\
& +\mu\int_{\partial\Omega}\sigma^{m}\dot{u}^{j}\curl u\nabla_j^{\perp}u\cdot nds-\mu\int\sigma^{m}\nabla\dot{u}^{j}\cdot\nabla_j^{\perp}u\curl udx\\
& =\mu\int_{\partial\Omega}\sigma^{m}\curl u_t\dot{u}\cdot\omega ds+\mu\int_{\partial\Omega}\sigma^{m}\div(u\curl u)\dot{u}\cdot\omega ds+\mu\int_{\partial\Omega}\sigma^{m}\dot{u}^{j}\curl u\nabla_j^{\perp}u\cdot nds \\
&\quad-\mu\int\sigma^{m}(\curl\dot{u})^{2}dx+\mu\int\sigma^{m}\nabla\dot{u}^{j}\cdot\nabla_j^{\perp}u\curl udx
\ea\ee
Now we have to estimate these three boundary terms in \eqref{ax4XG}. Notice that $\curl u=2(\kappa-\vartheta)u\cdot\omega$, a direct calculation together with \eqref{tb90} gives
\be\la{ax4XG1}\ba &\int_{\partial\Omega}\sigma^{m}\curl u_t\dot{u}\cdot\omega ds+\int_{\partial\Omega}\sigma^{m}\div(u\curl u)\dot{u}\cdot\omega ds+\int_{\partial\Omega}\sigma^{m}\dot{u}^{j}\curl u\nabla_j^{\perp}u\cdot nds\\
 &=2\int_{\partial\Omega}\sigma^{m}(\kappa-\vartheta) (\dot{u}\cdot\omega)^2ds-2\int_{\partial\Omega}\sigma^{m}(\kappa-\vartheta) \dot{u}\cdot\omega\,u\cdot\nabla (u\cdot\omega)ds\\
 &\quad+2\int_{\partial\Omega}\sigma^{m}(\kappa-\vartheta) \dot{u}\cdot\omega(u\cdot\nabla \omega\cdot u)ds+2\int_{\partial\Omega}\sigma^{m}(\kappa-\vartheta)\div u \, u\cdot\omega\,\dot u\cdot\omega ds \\
 &\quad+2\int_{\partial\Omega}\sigma^{m} u\cdot\nabla(\kappa-\vartheta) u\cdot\omega\,\dot{u}\cdot\omega ds+2\int_{\partial\Omega}\sigma^{m}(\kappa-\vartheta) \dot{u}\cdot\omega\,u\cdot\nabla (u\cdot\omega)ds\\
 &\quad +2\int_\Omega\sigma^{m}\nabla\left((\kappa-\vartheta)\,u\cdot\omega\, \dot{u}^{j}\right)\cdot\nabla_j^{\perp}u dx+2\int_\Omega\sigma^{m}(\kappa-\vartheta)\,u\cdot\omega\, \dot{u}\cdot\nabla^{\perp}\div u dx\\
 &=2\int_{\partial\Omega}\sigma^{m}(\kappa-\vartheta) (\dot{u}\cdot\omega)^2ds\color{black}-2\int_{\partial\Omega}\sigma^{m}(\kappa-\vartheta) \dot{u}\cdot\omega(u\cdot\nabla \omega\cdot u)ds\\
 &\quad-2\int_\Omega\sigma^{m} \nabla^\perp\left( u\cdot\omega\,\dot{u}\cdot\omega\,|u|\right)\cdot\nabla(\kappa-\vartheta) dx+2\int_\Omega\sigma^{m}\nabla\left((\kappa-\vartheta)\,u\cdot\omega\, \dot{u}^{j}\right)\cdot\nabla_j^{\perp}u dx\\
 &\quad +2\int_\Omega\sigma^{m}\div u \,\curl \left((\kappa-\vartheta)\,u\cdot\omega\, \dot{u})\right) dx\\
&\leq C\sigma^m(\|\nabla u\|_{L^2}^2\|\dot{u}\|_{L^6}+\|\nabla u\|_{L^2}^2\|\nabla\dot{u}\|_{L^2}+\|\nabla u\|_{L^4}^2\|\dot{u}\|_{L^2}+\|\nabla u\|_{L^4}^2\|\nabla\dot{u}\|_{L^2})\\
&\leq \frac{\delta}{2}\sigma^m\|\nabla\dot{u}\|_{L^2}^2+C\sigma^m\|\nabla u\|_{L^4}^4+C\sigma^m\|\nabla u\|_{L^2}^4.
\ea\ee
Consequently,
\be\la{ax4}\ba J_2\leq-\mu\sigma^m\|\curl \dot{u}\|_{L^2}^2+\frac{\delta}{2}\sigma^m\|\nabla\dot{u}\|_{L^2}^2+C\sigma^m(\|\nabla u\|_{L^4}^4+\|\nabla u\|_{L^2}^4).
\ea\ee
It follows from \eqref{ax1}, \eqref{ycax1}, \eqref{ax4} and Young's inequality that
\be\la{ax40}\ba
&\left(\frac{\sigma^{m}}{2}\|\rho^{\frac{1}{2}}\dot{u}\|_{L^2}^2\right)_t+(\lambda+2\mu)\sigma^{m}\|\div\dot{u}\|_{L^2}^2+\mu\sigma^{m}\|\curl\dot{u}\|_{L^2}^2\\
& \leq\frac{m}{2}\sigma^{m-1}\sigma'\|\rho\dot{u}\|_{L^2}^2-\left(\int_{\partial\Omega}\sigma^{m}\kappa  F|u|^{2}ds\right)_t +\delta\sigma^{m}\|\nabla\dot{u}\|_{L^2}^2+C\sigma^{m}\|\nabla u\|_{L^2}^4\\
&\quad+C\sigma^{m}(\|\rho\dot{u}\|_{L^2}^2\|\nabla u\|_{L^2}^2 +\|\rho\dot{u}\|_{L^2}^2+\|\nabla u\|_{L^2}^2+\|\nabla u\|_{L^4}^4+\|P-\bar{P}\|_{L^2}^2\|\nabla u\|_{L^2}^2)\\
&\quad + C\sigma^{m}\|P-\bar{P}\|_{L^2}^2 + Cm\sigma^{m-1}(\|\nabla u\|_{L^2}^2+\|\nabla u\|_{L^2}^4+\|\rho\dot{u}\|_{L^2}^2),
\ea\ee
together with \eqref{tb11}, which implies
\be\la{ax401}\ba
&\left(\sigma^{m}\|\rho^{\frac{1}{2}}\dot{u}\|_{L^2}^2\right)_t+(\lambda+2\mu)\sigma^{m}\|\div\dot{u}\|_{L^2}^2+\mu\sigma^{m}\|\curl\dot{u}\|_{L^2}^2\\
& \leq Cm\sigma^{m-1}\sigma'\|\rho^{\frac{1}{2}}\dot{u}\|_{L^2}^2-\left(2\int_{\partial\Omega}\sigma^{m}\kappa  F|u|^{2}ds\right)_t +C\sigma^{m}\|\rho^{\frac{1}{2}}\dot{u}\|_{L^2}^2\|\nabla u\|_{L^2}^2\\
& \quad +C\sigma^{m}(\|\rho^{\frac{1}{2}}\dot{u}\|_{L^2}^2+\|\nabla u\|_{L^2}^2+\|\nabla u\|_{L^4}^4+\|\nabla u\|_{L^2}^4+\|P-\bar{P}\|_{L^2}^2\|\nabla u\|_{L^2}^2)\\
&\quad + C\sigma^{m}\|P-\bar{P}\|_{L^2}^2 + Cm\sigma^{m-1}(\|\nabla u\|_{L^2}^2+\|\nabla u\|_{L^2}^4+\|\rho\dot{u}\|_{L^2}^2),
\ea\ee
if we choose $\delta$ small enough.

Integrating over $(0,T]$, by \eqref{tb11} and Lemma \ref{le3}, we have, for $m\geq1$,
\be\la{ax402}\ba
&\sigma^{m}\|\rho^{\frac{1}{2}}\dot{u}\|_{L^2}^2+\int_0^T\sigma^{m}\|\nabla\dot{u}\|_{L^2}^2dt\\
& \leq Cm\int_0^{\sigma(T)}\sigma^{m-1}\|\rho^{\frac{1}{2}}\dot{u}\|_{L^2}^2dt+\frac{\sigma^{m}}{2}\|\rho^{\frac{1}{2}}\dot{u}\|_{L^2}^2+C\sigma^{m}\|\nabla u\|_{L^2}^4+C\sigma^{m}\|\nabla u\|_{L^2}^2\\
&\quad+C\int_0^T\sigma^{m}\|\rho^{\frac{1}{2}}\dot{u}\|_{L^2}^2\|\nabla u\|_{L^2}^2dt+C\int_0^T\sigma^{m}\|\rho^{\frac{1}{2}}\dot{u}\|_{L^2}^2dt+C\int_0^T\sigma^{m}\|\nabla u\|_{L^2}^2dt\\
& \quad +C\int_0^T\sigma^{m}\|\nabla u\|_{L^4}^4dt+C\int_0^T\sigma^{m}\|\nabla u\|_{L^2}^4dt+C\int_0^T\sigma^{m}\|P-\bar{P}\|_{L^2}^2\|\nabla u\|_{L^2}^2dt\\
& \quad +C\int_0^T\sigma^{m}\|P-\bar{P}\|_{L^2}^2dt+C\int_0^Tm\sigma^{m-1}(\|\nabla u\|_{L^2}^2+\|\nabla u\|_{L^2}^4+\|\rho\dot{u}\|_{L^2}^2)dt.
\ea\ee
where we have taken advantage of the fact
\be\la{kfu1}\ba
\left|\int_{\partial\Omega}\sigma^{m}\kappa  F|u|^{2}ds\right|&\leq C\sigma^m(\|\nabla F\|_{L^2}\|\nabla u\|_{L^2}^2+\|F\|_{L^6}\|u\|_{L^3}\|\nabla u\|_{L^2})\\
&\leq\frac{\sigma^m}{2}\|\rho^{\frac{1}{2}}\dot{u}\|_{L^2}^2+C(\|\nabla u\|_{L^2}^4+\|\nabla u\|_{L^2}^2).
\ea\ee
Therefore,
\be\la{ax7}\ba
&\sigma^{m}\|\rho^{\frac{1}{2}}\dot{u}\|_{L^2}^2+\int_0^T\sigma^{m}\|\nabla\dot{u}\|_{L^2}^2dt\\
& \leq Cm\int_0^{\sigma(T)}\sigma^{m-1}\|\rho^{\frac{1}{2}}\dot{u}\|_{L^2}^2dt+C\sigma^{m}\|\nabla u\|_{L^2}^4+C\sigma^{m}\|\nabla u\|_{L^2}^2\\
&\quad+C\int_0^T\sigma^{m}\|\rho^{\frac{1}{2}}\dot{u}\|_{L^2}^2\|\nabla u\|_{L^2}^2dt+C\int_0^T\sigma^{m}\|\rho^{\frac{1}{2}}\dot{u}\|_{L^2}^2dt+C\int_0^T\sigma^{m}\|\nabla u\|_{L^2}^2dt\\
& \quad +C\int_0^T\sigma^{m}\|\nabla u\|_{L^4}^4dt+C\int_0^T\sigma^{m}\|\nabla u\|_{L^2}^4dt+C\int_0^T\sigma^{m}\|P-\bar{P}\|_{L^2}^2\|\nabla u\|_{L^2}^2dt\\
& \quad +C\int_0^T\sigma^{m}\|P-\bar{P}\|_{L^2}^2dt+C\int_0^Tm\sigma^{m-1}(\|\nabla u\|_{L^2}^2+\|\nabla u\|_{L^2}^4+\|\rho\dot{u}\|_{L^2}^2)dt.
\ea\ee

Now take $m=2$ in \eqref{ax7}, by \eqref{a16}, \eqref{th00jl} and \eqref{zz1}, we get \eqref{h15}.
\end{proof}
\begin{lemma}\la{zc1} Suppose $(\n,u)$ be a smooth solution of \eqref{a1}-\eqref{ch1}   on $\O \times (0,T] $, then there exists a positive constant $K$ depending only on $\mu,\,\,\lambda,\,\,\gamma,\,\,a,\,\,\hat{\rho},\,\,s,\,\,\Omega,$ and $M,$ such that
   \be\la{uv1}  \sup_{0\le t\le \si(T)}t^{1-s}\|\na
u\|_{L^2}^2+\int_0^{\si(T)}t^{1-s}\int\n|\dot u|^2dxdt\le
K(\hat{\rho},M), \ee
 \be\la{uv2}  \sup_{0\le t\le \si(T)}t^{2-s}\int\n|\dot u|^2dx+\int_0^{\si(T)}t^{2-s}\int|\nabla\dot{u}|^2dxdt\le
K(\hat{\rho},M). \ee
\end{lemma}
\begin{proof} Suppose $w_1(x,t)$ and $w_2(x,t)$ are the solutions of the following equations
\be\la{fcz1}\begin{cases}
  Lw_1=0, &x\in\Omega,\\  w_1(x,0)=w_{10}(x), &x\in\Omega,\\w_1\cdot n=0,\,\,\curl w_1=(\kappa-\vartheta) w_1\cdot\omega, &x\in\partial\Omega,
\end{cases}\ee
and
\be\la{fcz2}\begin{cases}
  Lw_2=-\nabla(P-\bar{P}), &x\in\Omega\\  w_2(x,0)=0, &x\in\Omega,\\w_2\cdot n=0,\,\,\curl w_2=(\kappa-\vartheta)  w_2\cdot\omega, &x\in\partial\Omega ,
\end{cases}\ee
respectively, where $Lf\triangleq\rho\dot{f}-\mu\Delta f-(\lambda+\mu)\nabla\div f .$

Recall the previous analysis in the proof of Lemma \ref{le3}, by the standard $L^{p}$-estimate for an elliptitic system and Sobolev's inequality, for any $2\leq p<\infty,$ we easily have

\be \la{xbh1} \ba
\|\nabla w_1\|_{L^{p}}\leq C\|w_1\|_{H^{2}}\leq C(\|\rho\dot{w}_1\|_{L^{2}}+\|\nabla w_1\|_{L^{2}}),
\ea\ee
\be \la{xbh2} \ba
\|\nabla F_{w_2}\|_{L^{p}}\leq C(\|\rho\dot{w}_2\|_{L^{p}}+\|\nabla w_2\|_{L^{p}}),
\ea\ee
\be \la{xbh3} \ba
\|F_{w_2}\|_{L^{p}}&\leq C(\|\nabla F_{w_2}\|_{L^{2}}+\|F_{w_2}\|_{L^{2}})\\
& \le C (\|\rho\dot{w}_2\|_{L^{2}}+\|\nabla w_2\|_{L^{2}}+\|P-\bar{P}\|_{L^{2}}),
\ea\ee
\be \la{xbh4} \ba
\|\nabla w_2\|_{L^{p}}&\leq C\|\rho^{\frac{1}{2}}\dot{w}_2\|_{L^{2}}^{1-\frac{2}{p}}(\|\nabla w_2\|_{L^{2}}+\|P-\bar{P}\|_{L^{2}})^{\frac{2}{p}}+C(\|\nabla w_2\|_{L^{2}}+\|P-\bar{P}\|_{L^{p}}),
\ea\ee
where $F_{w_2}=(\lambda+2\mu)\div w_2-(P-\bar{P})  .$

Just as the proof of \eqref{a16} that
\be\la{xbh5}  \sup_{0\le t\le \si(T)}\int\rho|w_1|^2dx+\int_0^{\si(T)}\int|\nabla w_1|^2dxdt\le
C(\hat{\rho})\int|w_{10}|^{2}dx  ,\ee
and
\be\la{xbh6}  \sup_{0\le t\le \si(T)}\int\rho|w_2|^2dx+\int_0^{\si(T)}\int|\nabla w_2|^2dxdt\le
C(\hat{\rho})C_0  .\ee

We multiply \eqref{fcz1} by $w_{1t}$ and integrate the result over $\Omega.$ By \eqref{xbh1}, \eqref{xuv1},  Sobolev's and Young's inequalities, we get
\be \la{xbh9} \ba
&\left(\frac{\lambda+2\mu}{2}\int(\div w_1)^{2}dx + \frac{\mu}{2}\int(\curl w_1)^{2}dx-\int_{\partial\Omega}(\kappa-\vartheta) |w_1|^2ds\color{black}\right)_t + \int\rho|\dot{w}_1|^{2}dx\\
& =\int\rho\dot{w}_1\cdot(u\cdot\nabla w_1)dx \\
& \le C(\hat{\rho})\|\rho^{\frac{1}{2}}\dot{w}_1\|_{L^{2}}\|\nabla w_1\|_{L^{\frac{2(\nu+2)}{\nu}}}\|\rho^{\frac{1}{2+\nu}}u\|_{L^{2+\nu}} \\
& \le C(\hat{\rho})\|\rho^{\frac{1}{2}}\dot{w}_1\|_{L^{2}}\|\nabla^{2} w_1\|_{L^{2}}^{\frac{2}{2+\nu}}\|\nabla w_1\|_{L^{2}}^{\frac{\nu}{2+\nu}}+C(\hat{\rho})\|\rho^{1/2}\dot{w}_1\|_{L^2}\|\nabla w_1\|_{L^2} \\
& \le C(\hat{\rho},M)\left(\|\rho^{\frac{1}{2}}\dot{w}_1\|_{L^{2}}+\|\nabla w_1\|_{L^{2}}\right)^{\frac{4+\nu}{2+\nu}}\|\nabla w_1\|_{L^{2}}^{\frac{\nu}{2+\nu}}+C(\hat{\rho})\|\rho^{1/2}\dot{w}_1\|_{L^2}\|\nabla w_1\|_{L^2} \\
& \le \frac{1}{2}\|\rho^{\frac{1}{2}}\dot{w}_1\|_{L^{2}}^{2}+C(\hat{\rho},M)\|\nabla w_1\|_{L^{2}}^{2}.
\ea\ee
\eqref{xbh5} and Gronwall's inequality show that
\be \la{xbh10} \ba
\sup_{0\le t\le \si(T)}\|\nabla w_1\|_{L^{2}}^{2}+\int_0^{\sigma(T)}\int\rho|\dot{w}_1|^{2}dxdt\leq C(\hat{\rho},M)\|\nabla w_{10}\|_{L^{2}}^{2}  ,
\ea\ee
and
\be \la{xbh11} \ba
\sup_{0\le t\le \si(T)}t\|\nabla w_1\|_{L^{2}}^{2}+\int_0^{\sigma(T)}t\int\rho|\dot{w}_1|^{2}dxdt\leq C(\hat{\rho},M)\|w_{10}\|_{L^{2}}^{2}  .
\ea\ee

Obviously, operator $w_{10}\mapsto w_1(\cdot,t)$ is linear. One can deduce from (\ref{xbh10}) ,(\ref{xbh11}) and the standard Stein-Weiss interpolation argument \cite{bl} that for any $\theta\in [s,1],$
\be \la{xbh12} \ba
 \sup_{0\le t\le
\si(T)}t^{1-\theta}\|\nabla
w_1\|_{L^2}^2+\int_0^{\si(T)}t^{1-\theta}\int\n|\dot
{w_1}|^2dxdt\leq C(\hat{\rho},M)\| w_{10}\|_{\dot H^\theta}^2,
\ea\ee
where $C$ is independent of $\theta.$

 Multiplying (\ref{fcz2}) by $w_{2t} $ and integrate the resulting equality over $\Omega$, we have
\be \la{xbh13} \ba
&\left(\frac{\lambda+2\mu}{2}\int(\div w_2)^{2}dx+\frac{\mu}{2}\int(\curl w_2)^{2}dx-\int P\div w_2dx -\int_{\partial\Omega}(\kappa-\vartheta) |w_2|^2ds\color{black}\right)_t \\
&\quad +\int\rho|\dot{w}_2|^{2}dx\\
& =\int\rho\dot{w}_2\cdot(u\cdot\nabla w_2)dx-\int P_t\div w_2dx \\
& =\int\rho\dot{w}_2\cdot(u\cdot\nabla w_2)dx - \frac{1}{\lambda+2\mu}\int PF_{w_2}\div u dx- \frac{1}{\lambda+2\mu}\int P\nabla F_{w_2}\cdot u dx \\
&\quad- \frac{1}{2(\lambda+2\mu)}\int(P-\bar{P})^{2}\div u dx+\gamma\int P\div u\div w_2dx \\
& \leq C\|\rho^{\frac{1}{2}}\dot{w}_2\|_{L^{2}}\|\rho^{\frac{1}{2+\nu}}u\|_{L^{2+\nu}}(\|\nabla w_2\|_{L^{2}}+\|P-\bar{P}\|_{L^{2}})^{\frac{\nu}{2+\nu}}\|\rho^{\frac{1}{2}}\dot{w}_2\|_{L^{2}}^{\frac{2}{2+\nu}}\\
& \quad + C(\|\rho^{\frac{1}{2}}\dot{w}_2\|_{L^{2}}\|\rho^{\frac{1}{2+\nu}}u\|_{L^{2+\nu}}\|P-\bar{P}\|_{L^{2(2+\nu)/\nu}}+\|\rho^{\frac{1}{2}}\dot{w}_2\|_{L^{2}}\|\rho^{\frac{1}{2+\nu}}u\|_{L^{2+\nu}}\|\nabla w_2\|_{L^2})\\
& \quad +C(\|\nabla w_2\|_{L^{2}}+\|P-\bar{P}\|_{L^{2}})\|\nabla u\|_{L^2}+C(\|\rho^{\frac{1}{2}}\dot{w}_2\|_{L^{2}}\|\nabla u\|_{L^2}+\|\nabla w_2\|_{L^2}\|\nabla u\|_{L^2})\\
&\quad+C\|P-\bar{P}\|_{L^{2}}\|\nabla u\|_{L^2}+C\|\nabla u\|_{L^{2}}\|\nabla w_2\|_{L^2}\\
& \leq C\|\rho^{\frac{1}{2}}\dot{w}_2\|_{L^{2}}^{\frac{4+\nu}{2+\nu}}(\|\nabla w_2\|_{L^{2}}+\|P-\bar{P}\|_{L^{2}})^{\frac{\nu}{2+\nu}}+\frac{1}{4}\|\rho^{\frac{1}{2}}\dot{w}_2\|_{L^{2}}^2\\
&\quad+C(\|\nabla w_2\|_{L^2}^2+\|\nabla u\|_{L^2}^2+\|P-\bar{P}\|_{L^2}^{\frac{\nu}{2+\nu}}) \\
& \leq\frac{1}{2}\|\rho^{\frac{1}{2}}\dot{w}_2\|_{L^{2}}^{2}+C(\|\nabla w_2\|_{L^{2}}^{2}+\|\nabla u\|_{L^{2}}^{2}+\|P-\bar{P}\|_{L^{2}}^{\frac{\nu}{2+\nu}}),
\ea\ee
where we have taken advantage of \eqref{xbh2}, \eqref{xbh4}, \eqref{xuv1} and Young's inequality.
So
\be \la{xbh14} \ba
&\left((\lambda+2\mu)\int(\div w_2)^{2}dx+\mu\int(\curl w_2)^{2}dx-2\int P\div w_2dx-2\int_{\partial\Omega}(\kappa-\vartheta) |w_2|^2ds\color{black}\right)_t \\
&\quad +\int\rho|\dot{w}_2|^{2}dx\le C\left(\|\nabla w_2\|_{L^{2}}^{2}+\|\nabla u\|_{L^{2}}^{2}+\|P-\bar{P}\|_{L^{2}}^{\frac{\nu}{2+\nu}}\right) .
\ea\ee
By Gronwall's inequality, Lemma \ref{le2}, \eqref{th00jl} and \eqref{xbh6}, one has
\be \la{xbh15} \ba
\sup_{0\le t\le \si(T)}\|\nabla w_2\|_{L^{2}}^{2}+\int_0^{\sigma(T)}\int\rho|\dot{w}_2|^{2}dxdt\leq C,
\ea\ee
where we have used the inequality
$\|\nabla w_2\|_{L^2}\leq C(\|\div w_2\|_{L^2}+\|\curl w_2\|_{L^2})$, since $w_2\cdot n=0$ on $\partial\Omega$.
Taking $w_{10}=u_0 $ , then $w_1+w_2=u.$ \eqref{uv1} is obtained from (\ref{xbh12}) and (\ref{xbh15}) directly.

Next, we prove \eqref{uv2}. Taking $m=2-s$ in \eqref{ax7} and integrating over $(0,\sigma(T))$ instead of $(0,T)$, together with \eqref{uv1}, we get
\be \la{xbh16} \ba
&\sup_{0\le t\le  \si(T)}t^{2-s}\int\n|\dot u|^2dx+\int_0^{ \si(T)}t^{2-s}\|\nabla\dot{u}\|_{L^2}^2dt \\
& \le  C\int_0^{\si(T)}t^{2-s}\|\na u\|_{L^4}^4dt+C(\hat{\rho},M) . \ea\ee

By \eqref{h18}, \eqref{th00jl} and \eqref{uv1}, we get
\be \la{xbh17}\ba
&\int_0^{\sigma(T)}t^{2-s}\|\nabla u\|_{L^{4}}^{4}dt \\
& \le C\int_0^{\sigma(T)}t^{2-2s}\|\rho^{\frac{1}{2}}\dot{u}\|_{L^{2}}^{2}(\|\nabla u\|_{L^{2}}^{2}+\|P-\bar{P}\|_{L^{2}}^{2})dt+C\int_0^{\sigma(T)}t^{2-s}\|\nabla u\|_{L^{2}}^{4}dt\\
&\quad+C\int_0^{\sigma(T)}t^{2-s}\|P-\bar{P}\|_{L^{4}}^{4}dt \\
& \le C\int_0^{\sigma(T)}t^s(t^{1-s}\|\rho^{\frac{1}{2}}\dot{u}\|_{L^{2}}^{2})(t^{1-s}\|\nabla u\|_{L^{2}}^{2})dt+C\int_0^{\sigma(T)}t^{1-s}\|\nabla u\|_{L^{2}}^{2}t\|\nabla u\|_{L^{2}}^{2}dt\\
& \quad+C\int_0^{\sigma(T)}(t^{1-s}\|\rho^{\frac{1}{2}}\dot{u}\|_{L^{2}}^{2})(t\|P-\bar{P}\|_{L^{2}}^{2})dt+C(\hat{\rho})\\
& \leq C(\hat{\rho}, M) .
\ea\ee
As a result,
\be \la{xbh18} \ba
&\sup_{0\le t\le  \si(T)}t^{2-s}\int\n|\dot u|^2dx+\int_0^{ \si(T)}t^{2-s}\|\nabla u\|_{L^{2}}^{2}dt\leq C(\hat{\rho}, M).
\ea\ee
\end{proof}
\begin{lemma}\la{le5} Suppose $(\n,u)$ is a smooth solution  of
   \eqref{a1}-\eqref{ch1}     on $\O \times (0,T] $ satisfying \eqref{zz1}. Then there exists a positive constant $C(\hat\rho)$ depending only  on $\mu,$  $\lambda,$ $a,$ $\gamma,$ $s,$ $\hat\rho,$ $\Omega$ and $M$
 such that
  \be\la{h27}
  A_1(T)+A_2(T)\le C_0^{1/3},
  \ee
provided $C_0\leq\varepsilon_0$.
   \end{lemma}

\begin{proof} By (\ref{h18}),   \eqref{a16}, \eqref{th00jl} and  \eqref{zz1}, it is easy to check that
  \be\la{h99} \ba
  &  \int_0^{T}\sigma^2 \|\na u\|_{L^4}^4 dt\\
& \le  C \int_0^{T}\sigma^{2} \|\n^{\frac{1}{2}}  \dot u \|_{L^2}^2(\|\nabla u\|_{L^{2}}^{2}+\|P-\bar{P}\|_{L^{2}}^{2})dt +C\int_{0}^{T}\sigma^{2}\|\nabla u\|_{L^{2}}^{4}dt \\
&\quad+C\int_{0}^{T}\sigma^{2}\|P-\bar{P}\|_{L^{4}}^{4}dt \\
& \le C\int_{0}^{T}(\sigma\|\rho^{\frac{1}{2}}\dot{u}\|_{L^{2}}^{2})(\sigma\|\nabla u\|_{L^{2}}^{2})dt+C\int_{0}^{T}(\sigma\|\rho^{\frac{1}{2}}\dot{u}\|_{L^{2}}^{2})(\sigma\|P-\bar{P}\|_{L^{2}}^{2})dt+CC_0^{3/4}\\
& \le C\left[(A_1(T)+C_0^{1/2})\int_0^T\sigma\|\rho^{\frac{1}{2}}\dot{u}\|_{L^{2}}^{2}dt+C_0^{3/4}\right] \\
& \le C((A_1(T)+C_0^{1/2})A_1(T)+C_0^{3/4}) \\
& \le CC_0^{2/3},
 \ea \ee
which, together with \eqref{h14} and \eqref{h15} gives
\be\la{h991} \ba
A_1(T)+A_2(T)\leq C\left(C_0^{1/2}+\int_0^T\sigma\|\nabla u\|_{L^{3}}^{3}dt\right)   .
\ea \ee

In the following, we estimate $\int_0^T\sigma\|\nabla u\|_{L^{3}}^{3}dt .$

First, it follows from  \eqref{h18}, \eqref{a16} and \eqref{zz1} that

\be\la{h992} \ba
&\int_0^{\sigma(T)}\sigma\|\nabla u\|_{L^{3}}^{3}dt \\
& \le C\int_0^{\sigma(T)}\sigma\|\rho^{\frac{1}{2}}\dot{u}\|_{L^{2}}(\|\nabla u\|_{L^{2}}^{2}+\|P-\bar{P}\|_{L^{2}}^{2})dt+C\int_0^{\sigma(T)}\sigma\|\nabla u\|_{L^{2}}^{3}dt\\
&\quad+C\int_0^{\sigma(T)}\sigma\|P-\bar{P}\|_{L^{3}}^{3}dt\\
& \le C\int_0^{\sigma(T)}(t^{\frac{1-s}{2}}\|\nabla u\|_{L^{2}})\|\nabla u\|_{L^{2}}(\sigma\|\rho^{\frac{1}{2}}\dot{u}\|_{L^{2}}^{2})^{\frac{1}{2}}dt+CC_0^{3/4}\\
& \quad+C\int_0^{\sigma(T)}\sigma^{\frac{1}{2}}\|P-\bar{P}\|_{L^{2}}^{2}(\sigma\|\rho^{\frac{1}{2}}\dot{u}\|_{L^{2}}^{2})^{\frac{1}{2}}dt \\
& \le C\left(\sup_{0\le t\le \sigma(T)}t^{1-s}\|\nabla u\|_{L^{2}}^{2}\right)^{\frac{1}{2}}\left(\int_0^{\sigma(T)}\|\nabla u\|_{L^{2}}^{2}dt\right)^{\frac{1}{2}}\left(\int_0^{\sigma(T)}\sigma\|\rho\dot{u}\|_{L^{2}}^{2}dt\right)^{\frac{1}{2}} \\
& \quad+C\left(\int_0^{\sigma(T)}\sigma\|P-\bar{P}\|_{L^{2}}^{4}dt\right)^{\frac{1}{2}}\left(\int_0^{\sigma(T)}\sigma\|\rho\dot{u}\|_{L^{2}}^{2}dt\right)^{\frac{1}{2}}+CC_0^{3/4}\\
& \le CC_0^{2/3}.
\ea \ee

On the other hand, by Young's inequality, \eqref{a16} and \eqref{h99}, we get
\be\la{h993} \ba
&\int_{\sigma(T)}^T\sigma\|\nabla u\|_{L^{3}}^{3}dt \\
& \le \int_{\sigma(T)}^T\sigma\|\nabla u\|_{L^{4}}^{4}dt+\int_{\sigma(T)}^T\sigma\|\nabla u\|_{L^{2}}^{2}dt\\
& \le CC_0^{2/3}.
\ea \ee
Together with \eqref{h992} and \eqref{h993}, it follows from \eqref{h991} that \eqref{h27} is true, provided $C_0\leq\varepsilon_0\triangleq C^{-6}$.
\end{proof}

Motivated by previous study on the two-dimensional Stokes approximation equations \cite{hlx1} or \cite{lx}, we can proceed to derive a uniform (in time) upper bound for the density, which is the key quantity to get all the higher order estimates and thus to extend the classical solution globally.

\begin{lemma}\la{le7}
There exists a positive constant
   $\ve$
    depending    on  $\mu,$  $\lambda,$ $\ga,$  $a,$ $\hat\rho,$  $s,$ $\Omega$, and $M$  such that,
    if  $(\rho,u)$ is a smooth solution  of
   \eqref{a1}-\eqref{ch1}     on $\O \times (0,T] $
   satisfying \eqref{zz1} and the assumptions in Theorem \ref{th1}, then
      \be\la{lv102}\sup_{0\le t\le T}\|\n(t)\|_{L^\infty}  \le
\frac{7\hat \n }{4}  ,\ee
      provided $C_0\le \ve.$ Moreover, if $C_0\le \ve$, then there exists some positive constant $\tilde{C}(T)$ depending only on $T,$ $\mu,$ $\lambda,$ $\ga,$  $a,$ $\hat\rho,$  $s,$ $\Omega$, $M$ and $\vartheta$ such that for $(x,t)\in\Omega\times(0,T)$
\be\la{le7xz}\rho(x,t)\geq\tilde{C}(T)\inf_{x\in\Omega}\rho_0(x).\ee

   \end{lemma}

\begin{proof}
  First, we notice that $(\ref{a1})_1$  can be rewritten as
  \be \la{z.3}  D_t \n=g(\n)+b'(t), \ee where \bnn
D_t\n\triangleq\n_t+u \cdot\nabla \n ,\quad
g(\n)\triangleq-\frac{\rho(P-\bar{P})}{\lambda+2\mu}  ,
\quad b(t)\triangleq -\frac{1}{\lambda+2\mu} \int_0^t\rho F dt. \enn
For $t\in[0,\sigma(T)],$ using the Gagliardo-Nirenberg's inequality  \eqref{g2} for $q=2$ and $r=12,$  we deduce from \eqref{uv1}, \eqref{h19} and Lemma \ref{uup1} that for all $0\leq t_1\leq t_2\leq\sigma(T)$,
\be \la{xbh19} \ba
&|b(t_2)-b(t_1)|  \\
& =\frac{1}{\lambda+2\mu}\left|\int_{t_1}^{t_2}\rho Fdt\right|\\
& \le C(\hat{\rho})\int_0^{\sigma(T)}\|F(\cdot,t)\|_{L^{\infty}}dt\\
& \le C(\hat{\rho})\int_0^{\sigma(T)}\|F\|_{L^{2}}^{5/11}\|\nabla F\|_{L^{12}}^{6/11}dt \\
& \le C(\hat{\rho})\int_0^{\sigma(T)}(\|\nabla u\|_{L^2}^{5/11}+\|P-\bar{P}\|_{L^2}^{5/11})(\|\nabla\dot{u}\|_{L^2}^{6/11}+\|\nabla u\|_{L^2}^{12/11}+\|\nabla u\|_{L^2}^{6/11})dt\\
& \quad+C(\hat{\rho})\int_0^{\sigma(T)}(\|\nabla u\|_{L^2}^{5/11}+\|P-\bar{P}\|_{L^2}^{5/11})\|P-\bar{P}\|_{L^2}^{6/11}dt\\
& \le C(\hat{\rho})\int_0^{\sigma(T)}(\|\nabla u\|_{L^2}^{5/11}\|\nabla\dot{u}\|_{L^2}^{6/11}+\|\nabla u\|_{L^2}^{17/11}+\|\nabla u\|_{L^2}+\|\nabla u\|_{L^2}^{5/11}\|P-\bar{P}\|_{L^2}^{6/11})dt\\
& \quad+C(\hat{\rho})\int_0^{\sigma(T)}(\|P-\bar{P}\|_{L^2}^{5/11}(\|\nabla\dot{u}\|_{L^2}^{6/11}+\|\nabla u\|_{L^2}^{12/11}+\|\nabla u\|_{L^2}^{6/11})+\|P-\bar{P}\|_{L^2})dt\\
& \le C(\hat{\rho})\int_0^{\sigma(T)}(t^{1-s}\|\nabla u\|_{L^2}^2)^{5/22}t^{-\frac{17-5s}{22}}(t^2\|\nabla\dot{u}\|_{L^2}^2)^{\frac{3}{11}}dt+C(\hat{\rho})\left(\int_0^{\sigma(T)}\|\nabla u\|_{L^2}^2dt\right)^{\frac{17}{22}}\\
&\quad+C(\hat{\rho})\left(\int_0^{\sigma(T)}\|\nabla u\|_{L^2}^2dt\right)^{3/11}+\int_0^{\sigma(T)}(t^{1-s}\|\nabla u\|_{L^2}^2)^{5/22}t^{-\frac{11-5s}{22}}(t\|P-\bar{P}\|_{L^2}^2)^{\frac{3}{11}}dt\\
& \quad+C(\hat{\rho})\left(\int_0^{\sigma(T)}t^2\|\nabla\dot{u}\|_{L^2}^{2}dt\right)^{3/11}\left(\int_0^{\sigma(T)}t^{-\frac{3}{4}}dt\right)^{8/11}+CC_0^{1/4}\\
& \le C(\hat{\rho},M)C_0^{1/11},
\ea\ee
Choosing $N_1=0,\,N_0=C(\hat{\rho},M)C_0^{1/11},\,\zeta_0=\hat{\rho}$ in Lemma  \ref{le1} and using \eqref{z.3}, \eqref{xbh19} give
   \be\la{a103}\sup_{t\in
[0,\si(T)]}\|\n\|_{L^\infty} \le \hat\rho
+C(\hat\rho,M)C_0^{1/11}\le\frac{3 \hat\n  }{2},\ee
 provided $$C_0\le \ve_1\triangleq\min\{1, (\hat\rho/(2C(\hat\rho,M)))^{11}\}. $$

On the other hand, for $t\in[\sigma(T),T],\,\,\sigma(T)\le t_1\le t_2\le T ,$ it follows from Lemmas \ref{le3}, \ref{uup1}, \ref{le2} and \eqref{zz1} that
\be \la{xbh20} \ba
&|b(t_2)-b(t_1)|  \\
&\le C(\hat{\rho})\int_{t_1}^{t_2}\|F\|_{L^{\infty}}dt \\
&\le \frac{a\hat{\rho}^{\gamma+1}}{2(\lambda+2\mu)}(t_2-t_1)+C(\hat{\rho})\int_{t_1}^{t_2}\|F\|_{L^{\infty}}^{3}dt \\
&\le \frac{a\hat{\rho}^{\gamma+1}}{2(\lambda+2\mu)}(t_2-t_1)+C(\hat{\rho})\int_{t_1}^{t_2}\|F\|_{L^2}\|\nabla F\|_{L^4}^{2}dt\\
& \le \frac{a\hat{\rho}^{\gamma+1}}{2(\lambda+2\mu)}(t_2-t_1)+C(\hat{\rho})C_0^{1/6}\int_{t_1}^{t_2}\|\nabla\dot{u}\|_{L^2}^2dt+CC_0^{1/2} \\
& \le \frac{a\hat{\rho}^{\gamma+1}}{\lambda+2\mu}(t_2-t_1)+C(\hat{\rho})C_0^{1/2}  ,
\ea\ee
which implies that one can choose $N_1$ and $N_0$ in \eqref{a100} as $N_1=\frac{a\hat{\rho}^{\gamma+1}}{\lambda+2\mu},\,\,N_0=C(\hat{\rho})C_0^{1/2}.$

Hence, we set $\zeta_0= \frac{3\hat{\rho}}{2} $ in (\ref{a101}) since for all $  \zeta \geq\zeta_0=\frac{3\hat{\rho}}{2},$
$$ g(\zeta)=-\frac{ a\zeta^{\gamma+1}-\zeta\bar{P}}{\lambda+2\mu}\le -\frac{a\hat{\rho}^{\gamma+1}}{2(\lambda+2\mu)}= -N_1. $$
As a result, Lemma \ref{le1} and \eqref{xbh20} lead to
\be\la{a102} \sup_{t\in
[\si(T),T]}\|\n\|_{L^\infty}\le  \frac{3\hat{\rho}}{2}
  +C(\hat{\rho})C_0^{1/2} \le
\frac{7\hat\rho }{4} ,\ee provided
\be \la{xbh21} \ba C_0\le
\ve\triangleq\min\{\ve_0,\ve_1, \ve_2 \}, \quad\mbox{ for
}\ve_2\triangleq ({ \hat\rho }/{4C(\hat\rho) })^{2}.
\ea\ee
Combining (\ref{a103}) and (\ref{a102}), we get \eqref{lv102}.

It remains to prove \eqref{le7xz}. If $\inf\limits_{x\in\Omega}\rho_0(x)=0,$ \eqref{le7xz} clearly holds. Assume that $\inf\limits_{x\in\Omega}\rho_0(x)>0,$
by \eqref{z.3}, We have
\bnn\ba
(\lambda+2\mu)D_t\rho+\rho(P-\bar{P})+\rho F=0.
\ea\enn
A simple computation shows
\bnn\ba
(\lambda+2\mu)D_t\rho^{-1}-\rho^{-1}(P-\bar{P}+F)=0.
\ea\enn
which yields that
\bnn\ba
D_t\rho^{-1}\leq C\rho^{-1}(|F|+1).
\ea\enn
Combining this with Gronwall's inequality, \eqref{xbh19} and \eqref{xbh20} gives \eqref{le7xz} and finishes
the proof of Lemma \ref{le7}.
\end{proof}
\section{\la{se5} A priori estimates (II): higher order estimates }

From now on, assume that the initial energy $C_0$  always meets the condition (\ref{xbh21}) and the positive
constant $C $ may depend on \bnn  T,\,\, \| g\|_{L^2},  \,\,\|\na u_0\|_{H^1},\,\,
    \|\rho_0\|_{W^{2,q}}  ,  \,\, \|P(\rho_0)\|_{W^{2,q}} , \,\,\enn
besides $\mu$, $\lambda$, $a$, $\ga$, $\hat\rho$,  $s,$  $\Omega$, and $
M,$ where $g\in L^2(\Omega)$ is as in \eqref{dt3}.

Next, we will derive important estimates on the spatial gradient of
the smooth solution $(\rho,u)$.

\begin{lemma}\la{xle1}
For $2\leq p<\infty$, there exists a positive constant $C,$ such that
    \be\label{cxb2}
\sup_{0\le t\le T}(\|\rho^{\frac{1}{2}}\dot{u}\|_{L^2}+\|u\|_{H^2}+\|\rho\|_{W^{1,p}}) + \int_0^T\|\nabla u\|_{L^\infty}+\|\nabla^{2}u\|_{L^p}^2+\|\dot{u}\|_{H^1}^{2}dt\le C,\ee
\end{lemma}

\begin{proof} Taking $s=1$ in (\ref{uv1}) along with
(\ref{h27}) gives
     \be\la{cxb3}
  \sup_{t\in[0,T]}\|\nabla u\|_{L^2}^2 + \int_0^{T}\int\rho|\dot{u}|^2dxdt
  \le C ,
  \ee
Combing  \eqref{cxb3} and \eqref{h18}, we obtain
\be\la{cxb4}
 \int_0^{T}\int\|\nabla u\|_{L^4}^4dt\le C.
\ee
Choosing $m=0$ in \eqref{ax401},  we have
\be\la{cxb5} \ba
&\left(\|\rho^{\frac{1}{2}}\dot{u}\|_{L^2}^2\right)_t+(\lambda+2\mu)\|\div\dot{u}\|_{L^2}^2+\mu\|\curl\dot{u}\|_{L^2}^2\\
& \leq-\left(2\int_{\partial\Omega}\kappa  F|u|^{2}ds\right)_t +C\|\rho^{\frac{1}{2}}\dot{u}\|_{L^2}^2\|\nabla u\|_{L^2}^2+C\|\rho^{\frac{1}{2}}\dot{u}\|_{L^2}^2\\
& \quad +C(\|\nabla u\|_{L^2}^2+\|P-\bar{P}\|_{L^2}^2)+C\|\nabla u\|_{L^4}^4
 .\ea\ee

Considering the compatibility condition \eqref{dt3}, we define
\be \la{cxb6}\sqrt{\n} \dot u(x,t=0)=\sqrt{\rho_{0}}g.\ee
Integrating \eqref{cxb5} over $[0,T],$ it follows from \eqref{tb11}, \eqref{kfu1} and \eqref{cxb6} that
\be \la{cxb7}\ba
&\|\rho^{\frac{1}{2}}\dot{u}\|_{L^2}^{2}+ \int_0^T\|\nabla\dot{u}\|_{L^2}^2dt  \\
&\le C\|\nabla u\|_{L^2}^4+C\|\nabla u\|_{L^2}^2+C\int_0^T\|\rho^{\frac{1}{2}}\dot{u}\|_{L^2}^2\|\nabla u\|_{L^2}^2dt+C\int_0^T\|\rho^{\frac{1}{2}}\dot{u}\|_{L^2}^2dt\\
& \quad +C\int_0^T(\|\nabla u\|_{L^2}^2+\|\nabla u\|_{L^2}^4+\|P-\bar{P}\|_{L^2}^2)dt+C\int_0^T\|\nabla u\|_{L^4}^4dt+C,
\ea\ee
together with \eqref{cxb3} and \eqref{cxb4}, we get
\be \la{cxb8}\ba
\sup_{t\in[0,T]}\|\rho^{\frac{1}{2}}\dot{u}\|_{L^2}^{2}+\int_0^T\|\nabla\dot{u}\|_{L^2}^{2}dt\le C  .
\ea\ee
In the following, we get \eqref{cxb2} by using Lemma 5 as in
\cite{hlx}. For $ p\geq 2,$ a simple computation shows that \bnnn \ba
& (|\nabla\rho|^p)_t + \text{div}(|\nabla\rho|^pu)+ (p-1)|\nabla\rho|^p\text{div}u  \\
 &+ p|\nabla\rho|^{p-2}(\nabla\rho)^t \nabla u (\nabla\rho) +
p\rho|\nabla\rho|^{p-2}\nabla\rho\cdot\nabla\text{div}u = 0.\ea
\ennn
As a result, \be\la{cxb9}\ba
\partial_t\norm[L^p]{\nabla\rho}&\le
 C(1+\norm[L^{\infty}]{\nabla u} )
\norm[L^p]{\nabla\rho} +C\|\na^2u\|_{L^p}\\ &\le
 C(1+\norm[L^{\infty}]{\nabla u} )
\norm[L^p]{\nabla\rho} +C\|\n\dot u\|_{L^p}, \ea\ee
due to
\be\la{cxb10}\ba
\|\na^2 u\|_{L^p}\le   C\left(\|\n\dot u\|_{L^p}+ \|\nabla
P \|_{L^p}\right),
\ea\ee
which follows from the standard
$L^p$-estimate for the following elliptic system:
\bnn\begin{cases} -\mu\Delta
u-(\mu+\lambda)\na {\rm div}u=-\n \dot u-\na P, &x\in\Omega,\\ u\cdot n=0, \,\,\curl u=2(\kappa-\vartheta) u\cdot\omega, &x\in\partial\Omega .
\end{cases}\enn
We deduce from Gagliardo-Nirenberg's inequality, \eqref{cxb3} and \eqref{h19} that, for any $q>2$,
\be\la{cxb11}\ba
&\|\div u\|_{L^\infty}+\|\curl u\|_{L^\infty}\\
&\leq C(\|F\|_{L^\infty}+\|P-\bar{P}\|_{L^\infty})+\|\curl u\|_{L^\infty} \\
&\le C\|F\|_{L^2}^{\frac{q-2}{2(q-1)}}\|\nabla F\|_{L^q}^{\frac{q}{2(q-1)}}+\|P-\bar{P}\|_{L^\infty}+C\|\curl u\|_{L^2}^{\frac{q-2}{2(q-1)}}\|\nabla \curl u\|_{L^q}^{\frac{q}{2(q-1)}}+C \\
&\le C+C\|\nabla F\|_{L^q}^{\frac{q}{2(q-1)}}+C\|\nabla\curl u\|_{L^q}^{\frac{q}{2(q-1)}} \\
&\le C+C\|\rho\dot{u}\|_{L^q}^{\frac{q}{2(q-1)}}.
\ea\ee

By Lemma \ref{le9}, it yields that
\be\la{cxb12}\ba
\|\na u\|_{L^\infty } &\le C\left(\|{\rm div}u\|_{L^\infty }+
\|\curl u\|_{L^\infty } \right)\ln(e+\|\na^2 u\|_{L^q }) +C\|\na
u\|_{L^2} +C \\
&\le C\left(1+\|\rho\dot{u}\|_{L^q}^{\frac{q}{2(q-1)}}\right)\ln(e+\|\rho\dot u\|_{L^q } +\|\na \rho\|_{L^q})+C \\
&\le C(1+\|\rho\dot{u}\|_{L^q})\ln(e+\|\na \rho\|_{L^q}) .
\ea\ee
Consequently, for any $p>2,$ \eqref{cxb9} becomes
\be\la{cxb13}\ba
\frac{d}{dt}(e+\|\nabla\rho\|_{L^p})&\leq C(1+\|\rho\dot{u}\|_{L^p})(e+\|\nabla\rho\|_{L^p})\ln(e+\|\nabla\rho\|_{L^p})+C\|\rho\dot{u}\|_{L^p}.
\ea\ee

We can rewrite \eqref{cxb13} as
\be\la{cxb14}\ba
&\frac{d}{dt}(\ln(e+\|\nabla\rho\|_{L^p}))\\
&\leq C(1+\|\rho\dot{u}\|_{L^p})\ln(e+\|\nabla\rho\|_{L^p})+C\|\rho\dot{u}\|_{L^p}\\
&\le C(1+\|\nabla\dot{u}\|_{L^2}+\|\nabla u\|_{L^2}^{2})\ln(e+\|\nabla\rho\|_{L^p})+C(\|\nabla\dot{u}\|_{L^2}+\|\nabla u\|_{L^2}^{2}),
\ea\ee
where in the last inequality, we have used \eqref{tb90}.

By Gronwall's inequality,  \eqref{cxb3} and \eqref{cxb8}, we have
\be\la{cxb15}\ba
\sup_{0\leq t\leq T}\|\nabla\rho\|_{L^{p}}\leq C ,
\ea\ee
for any $p\geq 2 .$

Furthermore, by \eqref{cxb8}, \eqref{tb90} and \eqref{cxb11}
\be\la{cxb16}\ba
\int_0^T\|\nabla u\|_{L^\infty}dt\leq C,\,\,\int_0^T\|\nabla^{2} u\|_{L^p}^{2}dt\leq C\,\, and\,\,\sup_{0\leq t\leq T}\|\nabla^{2}u\|_{L^{2}}\leq C.
\ea\ee

We complete the proof of Lemma \ref{xle1}.
\end{proof}

The following Lemmas \ref{xle2}--\ref{xle5} will deal with the
higher order estimates of a local classical solution to be a global
one. The proofs are similar to the ones in \cite{hx2,hlma}.  For completeness, we sketched them here .
\begin{lemma}\la{xle2}
 The following estimates hold:
\be\la{cxb17}\ba
\sup_{0\le t\le T}\|\rho^{\frac{1}{2}}u_t\|_{L^2}^2 + \ia\int|\nabla u_t|^2dxdt\le C,
\ea\ee
\be\la{cxb18}\ba
\sup_{0\le t\le T}(\|{\rho- \bar{\rho}}\|_{H^2} +
 \|{P- \bar{P}}\|_{H^2})\le C.
\ea\ee
\end{lemma}
\begin{proof} By Lemma \ref{xle1}, a simple computation yields that
\be\la{cxb19}\ba
\|\rho^{\frac{1}{2}}u_t\|_{L^2}^2&\le \|\rho^{\frac{1}{2}}\dot u \|_{L^2}^2+\|\n^{\frac{1}{2}}u\cdot\na u\|_{L^2}^2\\
&\le C+C\|u\|_{L^4}^2\|\nabla u\|_{L^4}^2 \\
&\le C+\|\nabla u\|_{L^2}^2\|u\|_{H^2}^2 \\
&\le C,
\ea\ee
and
\be \la{cxb20}\ba
 \int_0^T\|\nabla u_t\|_{L^2}^2dt &\le\int_0^T\|\nabla \dot
u\|_{L^2}^2dt + \int_0^T\|\nabla(u\cdot\nabla u)\|_{L^2}^2dt \\
&\le C+\int_0^T(\|\nabla u\|_{L^4}^4+\|u\|_{L^\infty}^2\|\nabla^{2}u\|_{L^2}^2)dt  \\
&\le C+C\int_0^T(\|\nabla^{2}u\|_{L^2}^4+\|\nabla u\|_{H^1}^{2}\|\nabla^{2}u\|_{L^2}^2)dt \\
&\le C ,
\ea\ee
 so we get \eqref{cxb17} .

 In the following, \eqref{cxb18} will be proved.

By \eqref{Pu2}, a simple computation shows that
\be \la{cxb22}\ba
&\frac{d}{dt}\left(\|\nabla^2P\|_{L^2}^2 +\|\nabla^2\rho\|_{L^2}^2\right)\\
&\le C(1+\|\nabla
u\|_{L^\infty}+\|\nabla^2u\|_{L^4}^2)(\|\nabla^2P\|_{L^2}^2 +\|\nabla^2\rho
\|_{L^2}^2) + C\|\nabla\dot{u}\|_{L^2}^2 + C ,
\ea\ee
 where we have used $(\ref{a1})_1$, Lemma \ref{xle1} and \eqref{tb102}.
Consequently, by Gronwall's inequality, \eqref{cxb22} and Lemma \ref{xle1} we obtain
  \bnn \sup_{0\le t\le T} {\left(\|\nabla^2P\|_{L^2}^2
+\norm[L^2]{\nabla^2 \rho } \right)}\le C. \enn Thus the proof of
Lemma \ref{xle2} is completed.
\end{proof}
\begin{lemma}\la{xle3}
It holds that: \be\la{cxb24}
   \sup\limits_{0\le t\le T}\left(
   \|\n_t\|_{H^1}+\|P_t\|_{H^1}\right)
    + \int_0^T\left(\|\n_{tt}\|_{L^2}^2+\|P_{tt}\|_{L^2}^2\right)dt
\le C,
  \ee
\be\la{cxb25}
   \sup\limits_{0\le t\le T}\si \|\nabla u_t\|_{L^2}^2
    + \int_0^T\si\|\rho^{\frac{1}{2}}u_{tt}\|_{L^2}^2dt
\le C.
  \ee
\end{lemma}
\begin{proof} (\ref{Pu2}) and Lemma \ref{xle1} show that
\be \la{cxb26}
\|P_t\|_{L^2}\le
C\|u\|_{L^\infty}\|\nabla P\|_{L^2}+C\|\nabla u\|_{L^2}\le C.
\ee
Differentiating (\ref{Pu2}) with respect to $x$ yields
\bnn
\nabla P_t+u\cdot\nabla\nabla P+\nabla u\cdot\nabla P+\ga \nabla P {\rm div}u+\ga P  \nabla{\rm div}u=0.
\enn
Consequently, by Lemmas \ref{xle1} and \ref{xle2}, we get
\bn\la{cxb27} \|\nabla P_t\|_{L^2}\le C\|u\|_{L^\infty}\|\nabla^2
P\|_{L^2}+C\|\nabla^{2} u\|_{L^2}\|\nabla P\|_{L^6}+C\|\nabla^2
u\|_{L^2}\le C.\en
Combining (\ref{cxb26}) with (\ref{cxb27}) implies
\bn \la{cxb28}\sup_{0\le t\le T}\|P_t\|_{H^1}\le C.
\en
Differentiating (\ref{Pu2}) with respect to $t$ yields
\be\la{cxb29} P_{tt} + \gamma P_t{\rm div}u +
\gamma P{\rm div}u_t + u_t\cdot\nabla P + u\cdot\nabla P_t = 0.
\ee
Multiplying \eqref{cxb29} by $P_{tt}$ and integrating the resulting equality over $\Omega\times[0,T],$ we deduce from \eqref{cxb28}, Lemmas \ref{xle1} and \ref{xle2} that
\be \la{cxb30}\ba
&\int_0^T\|P_{tt}\|_{L^2}^2dt \\
& = -\int_0^T\int\gamma P_{tt}P_t\div udxdt - \int_0^T\int\gamma P_{tt}P\div u_tdxdt  \\
& \quad - \int_0^T\int P_{tt}u_t\cdot\nabla Pdxdt - \int_0^T\int P_{tt}u\cdot\nabla P_tdxdt \\
& \le \frac{1}{2}\int_0^T\|P_{tt}\|_{L^2}^2dt + C\int_0^T(\|P_t\|_{L^4}^{4}+\|\nabla u\|_{L^4}^{4})dt+\int_0^T\|\nabla u_t\|_{L^2}^{2}dt \\
& \quad +\int_0^T\| P_{tt}\|_{L^2}\|u_t\|_{L^6}\|\nabla P\|_{L^3}dt+\int_0^T\| u\|_{L^\infty}\|\nabla P_t\|_{L^2}^{2}dt \\
& \le \frac{1}{2}\int_0^T\|P_{tt}\|_{L^2}^2dt + C\int_0^T(\|P_t\|_{H^1}^{4}+\|\nabla^{2} u\|_{L^2}^{4})dt+\int_0^T\|\nabla u_t\|_{L^2}^{2}dt \\
& \quad +\int_0^T\|\nabla u_t\|_{L^2}^{2}dt+\int_0^T\|\nabla u\|_{H^1}\|\nabla P_t\|_{L^2}^{2}dt \\
& \le \frac{1}{2}\int_0^T\|P_{tt}\|_{L^2}^2dt + C,
\ea\ee
where we have make use of Sobolev's inequality .

Therefore, it holds
$$\int_0^T\|P_{tt}\|_{L^2}^2dt \le C .$$
One can handle with $\n_t$ and
$\n_{tt}$ in a similar way. So we get (\ref{cxb24}).

Next, (\ref{cxb25}) will be proved. Differentiating  $(\ref{a1})_2$  with
respect to $t$ , then multiplying the resulting equation by
$u_{tt},$  we obtain
\be\la{cxb34} \ba
&\frac{d}{dt}\left((\lambda+2\mu)\int(\div u_t)^{2}dx+\mu\int(\curl u_t)^2dx\right)+2\int\rho|u_{tt}|^2dx  \\
&=\frac{d}{dt}\left(-\int\rho_t|u_t|^{2}dx-2\int\rho_tu\cdot\nabla u\cdot u_tdx+2\int P_t\div u_tdx+2\mu\int_{\partial\Omega}(\kappa-\vartheta) |u_t|^2ds\color{black}\right) \\
&\quad +\int\rho_{tt}|u_t|^{2}dx + 2\int(\rho_tu\cdot\nabla u)_t\cdot u_tdx-2\int\rho u_t\cdot\nabla u\cdot u_{tt}dx \\
&\quad -2\int\rho u\cdot\nabla u_t\cdot u_{tt}dx - 2\int P_{tt}\div u_tdx \\
&\triangleq\frac{d}{dt}I_0 + \sum\limits_{i=1}^5I_i .
\ea \ee
It follows from $(\ref{a1})_1,$  (\ref{cxb2}), (\ref{cxb17}), and
(\ref{cxb24}) that
\be \ba \la{cxb35}
I_0& =-\int_{
}\rho_t |u_t|^2 dx- 2\int_{ }\rho_t u\cdot\nabla u\cdot u_tdx+
2\int_{ }P_t {\rm div}u_tdx+2\mu\int_{\partial\Omega}(\kappa-\vartheta) |u_t|^2ds\\
&\le \left|\int_{ } {\rm
div}(\rho u)|u_t|^2dx\right|+C\norm[L^3]{\rho_t}\|\nabla u\|_{L^4}^2
\norm[L^6]{u_t}+C\|\nabla u_t\|_{L^2}\\
&\le C \int_{} \n |u||u_t||\nabla u_t| dx +C\|\nabla u_t\|_{L^2} \\
&\le C\|u\|_{L^6}\|\n^{1/2} u_t\|_{L^2}^{1/2}\|u_t\|_{L^6}^{1/2}\|\nabla
u_t\|_{L^2} +C\|\nabla u_t\|_{L^2}\\
&\le C\|\nabla u_t\|_{L^2}^{3/2}+C\|\nabla u_t\|_{L^2}\\
&\le \de\|\nabla u_t\|_{L^2}^2+C(\de),\ea\ee
Define
$$H(t)\triangleq(\lambda+2\mu)\int(\div u_t)^{2}dx+\mu\int(\curl u_t)^{2}dx ,$$
which satisfies
\be\ba\la{cxb40}
C^{-1}\|\nabla u_t\|_{L^2}\leq H(t)\leq C\|\nabla u_t\|_{L^2},
\ea\ee
since $u_t\cdot n = 0$ on $\partial\Omega.$
 \be \la{cxb36}\ba
|I_1|&=\left|\int_{ }\rho_{tt} |u_t|^2 dx\right|\\
& = \left|\int_{ }\div(\rho u)_t|u_t|^2 dx\right|\\
& = 2\left|\int_{ }(\rho_tu + \rho u_t)\cdot\nabla u_t\cdot u_tdx\right|\\
& \le  C\left(\norm[H^1]{\rho_t}\norm[H^2]{u}
  +\|\rho^{{1/2}}u_t\|_{L^2}^{1/2}\|\nabla u_t\|_{L^2}^{1/2}\right)\|\nabla u_t\|_{L^2}^2 \\
& \le C\|\nabla u_t\|_{L^2}^4 + C\\
& \le C\|\nabla u_t\|_{L^2}^2H(t) + C,
\ea \ee
   and
\be \la{cxb37}\ba
|I_2|&=2\left|\int_{ }\left(\rho_t u\cdot\nabla u \right)_t\cdot u_{t}dx\right|\\
& = 2\left|  \int_{ }\left(\rho_{tt} u\cdot\nabla u\cdot u_t +\rho_t
u_t\cdot\nabla u\cdot u_t+\rho_t u\cdot\nabla u_t\cdot
u_t\right)dx\right|\\
&\le2\norm[L^2]{\rho_{tt}}\norm[L^3]{u\cdot\nabla u}\norm[L^6]{u_t}+2\norm[L^2]{\rho_t}\|u_t\|_{L^6}^2\norm[L^6]{\nabla u} \\
&\quad+2\norm[L^3]{\rho_t}\norm[L^{\infty}]{u}\norm[L^2]{\nabla u_t}\norm[L^6]{u_t}\\
& \le C\norm[L^2]{\rho_{tt}}^2 + C\norm[L^2]{\nabla u_t}^2. \ea \ee
\be\ba\la{cxb38}
|I_3|+|I_4|&= 2\left| \int_{ }\rho u_t\cdot\nabla
u\cdot u_{tt} dx\right| +2\left| \int_{ }\rho u\cdot\nabla u_t\cdot
u_{tt} dx\right|\\& \le   C\|\n^{1/2}u_{tt}\|_{L^2}\left(
\|u_t\|_{L^6}\|\na u\|_{L^3}+\|u\|_{L^\infty}\|\na
u_t\|_{L^2}\right) \\& \le\delta\norm[L^2]{\rho^{{1/2}}u_{tt}}^2 +
C(\de)\norm[L^2]{\nabla u_t}^2 , \ea\ee
and
\be\ba\la{cxb39}
|I_5|&=2\left|\int_{ }P_{tt}{\rm div}u_tdx\right|\\&\le
2\norm[L^2]{P_{tt}}\norm[L^2]{{\rm div}u_t}\\& \le
C\norm[L^2]{P_{tt}}^2 + C\norm[L^2]{\nabla u_t}^2.
\ea\ee
Consequently,
\bnn\la{cxb41} \ba
&\frac{d}{dt}(\sigma H(t)-\sigma I_0)+\sigma\int\rho|u_{tt}|^{2}dx \\
&\le C(1+\|\nabla u_t\|_{L^2}^2)\sigma H(t)+C(1+\|\rho_{tt}\|_{L^2}^2+\|P_{tt}\|_{L^2}^2+\|\nabla u_t\|_{L^2}^2),
\ea \enn
Choosing $\delta>0$ suitably small, by Gronwall's inequality, \eqref{cxb17} and \eqref{cxb24}, we derive that
\bnn\la{cxb42} \ba
&\sup_{0\le t\le T}\sigma H(t)+\int_0^T\sigma\|\rho^{\frac{1}{2}}u_{tt}\|_{L^2}^2dt\le C .
\ea \enn
As a result, by \eqref{cxb40},
\bnn\la{cxb43} \ba
&\sup_{0\le t\le T}\sigma \|\nabla u_t\|_{L^2}^2+\int_0^T\sigma\|\rho^{\frac{1}{2}}u_{tt}\|_{L^2}^2dt\le C .
\ea \enn
This finishes the proof .
\end{proof}
\begin{lemma}\la{xle4}
For any $q\geq2,$ the following estimates hold:
\be\la{cxb44}\ba \sup_{t\in[0,T]}\left(\|\rho- \bar{\rho}\|_{W^{2,q}} +\|P-\bar{P}\|_{W^{2,q}}\right)\le C,\ea \ee
\be\la{cxb45}\ba\sup_{t\in[0,T]} \si \|\nabla u\|_{H^2}^2  +\ia \left(\|\nabla  u\|_{H^2}^2+\|\na^2
u\|^{p_0}_{W^{1,q}}+\si\|\na u_t\|_{H^1}^2\right)dt\le C,\ea \ee
where $p_0=(1,\frac{q}{q-1})$
\end{lemma}
 \begin{proof}
The standard $H^1$-estimate for elliptic system yields
\be\la{cxb46} \ba\|\nabla^2 u\|_{H^1} &\le C (\|\n \dot u\|_{H^1}+ \|P-\bar{P}\|_{H^2})+\|\nabla u\|_{L^2}\\
 &\le C+C \|\na  u_t\|_{L^2},
 \ea\ee
 where in the last inequality we have utilized \eqref{cxb2}, (\ref{cxb18}) as well as the following simple fact that
 \be \label{cxb47} \ba
  \|\nabla (\n \dot u) \|_{L^2}&\le
 \||\nabla \n | |  u_t|  \|_{L^2}+ \|\n \nabla   u_t  \|_{L^2}
 + \||\nabla \n|| u||\nabla u| \|_{L^2}\\ &\quad
 + \|\n|\nabla  u|^2\|_{L^2}
 + \|  \n |u || \nabla^2 u| \|_{L^2}\\&\le
 \|\nabla \n \|_{L^3} \|  u_t  \|_{L^6}+ C\| \nabla   u_t  \|_{L^2}
 + C\| \nabla \n\|_{L^3}\| u\|_{L^\infty}\|\nabla u \|_{L^6}\\
 &\quad + C\| \nabla  u\|_{L^3}\| \nabla  u\|_{L^6}
 + C\|     u \|_{L^\infty}\| \nabla^2 u  \|_{L^2}\\ &\le C+C\| \nabla   u_t  \|_{L^2}.
 \ea\ee
Then by (\ref{cxb46}),
(\ref{cxb2}), (\ref{cxb17}), and (\ref{cxb25}), we obtain
\be\la{cxb48}
\sup\limits_{0\le
t\le T}\si\|\nabla  u\|_{H^2}^2+\ia \|\nabla  u\|_{H^2}^2dt \le
 C.\ee
\eqref{cxb2} and \eqref{cxb24} show that
\be\la{cxb480}\ba
\|\na^2u_t\|_{L^2}
&\le C(\|(\rho\dot{u})_t\|_{L^2}+\|P_t\|_{H^1}+\|\nabla u_t\|_{L^2}) \\
&=C(\|\n  u_{tt}+\n_t u_t+\n_t u\cdot\nabla u + \n u_t\cdot\nabla u+\n u\cdot\nabla u_t\|_{L^2}+\|\nabla u_t\|_{L^2}+1)\\
&\le C\left(\|\n  u_{tt}\|_{L^2}+ \|\n_t\|_{L^3}
\|u_t\|_{L^6}+\|\n_t\|_{L^3}\| u\|_{L^\infty}\|\nabla
u\|_{L^6}\right)\\&\quad
  +C(\| u_t\|_{L^6}\|\nabla u\|_{L^3}+ \| u\|_{L^\infty}
  \|\nabla u_t\|_{L^2}+\|\nabla u_t\|_{L^2}+1)\\
&\le C\|\n  u_{tt}\|_{L^2} +C\|\nabla  u_t\|_{L^2}+C,
\ea \ee
where in the first inequality, we have used Lemma \ref{zhle} for the following elliptic system
\be\la{cxb49}\begin{cases}
  \mu\Delta u_t+(\lambda+\mu)\nabla\div u_t=(\rho\dot{u})_t+\nabla P_t , &x\in\Omega\\ u_t\cdot n=0,\,\,\curl u_t=(\kappa-\vartheta) u_t\cdot\omega, &x\in\partial\Omega .
\end{cases}\ee
Combining \eqref{cxb25}, it indicates
\be\la{cxb50}\ba
\int_0^T\sigma\|\nabla u_t\|_{H^1}^2dt\leq C .
\ea \ee
It follows from \eqref{cxb18} and \eqref{Pu2} that
\be\la{cxb51}\ba
(\|\na^2 P\|_{L^q})_t\le& C \|\na u\|_{L^\infty} \|\na^2 P\|_{L^q}   +C  \|\na^2 u\|_{W^{1,q}}   \\\le& C (1+\|\na u\|_{L^\infty} )\|\na^2 P\|_{L^q}    +C(1+ \|\na  u_t\|_{L^2}+C \| \na(\n
\dot u )\|_{L^{q}} ),
\ea\ee
 where in the second inequality we utilized the following simple fact that
\be \la{cxb52}\ba \|\na^2 u\|_{W^{1,q}}&\le C(\| \n\dot u \|_{W^{1,q}}+C\|P-\bar{P}\|_{W^{2,q}}+\|\nabla u\|_{L^2})\\
 &\le C (1+\| \n\dot u \|_{L^{q}}+\|\na  P\|_{L^{q}} + \| \na(\n\dot u )\|_{L^{q}}+\|\na^2  P\|_{L^{q}})\\
 &\le C(1 + \|\na  u_t\|_{L^2}+C \| \na(\n\dot u )\|_{L^{q}}+C\|\na^2  P\|_{L^{q}}),
 \ea\ee
 due to (\ref{cxb18}), (\ref{cxb2}), and (\ref{cxb47}).

On the other hand, observe that
\be\la{cxb53}\ba     \|\na(\n\dot u)\|_{L^q}
&\le C(\|\na \n\|_{L^q}\|\dot{u}\|_{H^1}+\|\na\dot u \|_{L^q})\\
&\le C(\|\dot{u}\|_{H^1}+\|\na u_t \|_{L^q}+\|\na(u\cdot \na u ) \|_{L^q})\\
&\le C(\|\dot{u}\|_{H^1}+\|\na u_t \|_{L^2}^{2/q}\|\na u_t \|_{H^1}^{1-2/q}+\|u \|_{L^\infty}\|\nabla^{2}u\|_{L^q}+\|\nabla u\|_{L^{2q}}^{2})\\
&\le C(\sigma^{-1/2}+\sigma^{-1/2}(\sigma\|\na u_t \|_{H^1}^2)^{(q-2)/2q}+\|\nabla^{2}u\|_{L^q}+1) ,
\ea \ee
due to \eqref{cxb2}.

By \eqref{cxb2} and \eqref{cxb50}, we get
\be\ba\la{cxb55}
\int_0^T\|\nabla(\rho\dot{u})\|_{L^q}^{p_0}dt\leq C .
\ea\ee
Hence, by Gronwall's inequality, we deduce from \eqref{cxb51}, \eqref{cxb55}, \eqref{cxb2} and \eqref{cxb17} that
\be\ba\la{cxb56}
\sup_{t\in[0,T]}\|\nabla^{2}P\|_{L^q}\leq C ,
\ea\ee
which along with \eqref{cxb17}, \eqref{cxb18}, \eqref{cxb52} and \eqref{cxb55} gives
\be\ba\la{cxb57}
\sup_{t\in[0,T]}\|P-\bar{P}\|_{W^{2,q}}+\int_0^T\|\nabla^{2}u\|_{W^{1,q}}^2dt\leq C .
\ea\ee
Similarly, one has
\bnn\sup\limits_{0\le t\le T}\|
\n-\bar{\rho}\|_{W^{2,q}} \le
 C,\enn
 which along with (\ref{cxb57})  gives (\ref{cxb44}). The proof of Lemma \ref{xle4}
is completed.
\end{proof}
\begin{lemma}\la{xle5} For any $q\geq 2$, it holds that
\be \la{cxb58}
\sup_{0\le t\le T}\si^2\left(\|\na u_t\|_{H^1}^{2}
 +\|\na u\|_{W^{2,q}}^{2}\right)
 +\int_{0}^T\si^2\|\nabla u_{tt}\|_{2}^2dt\le C .
 \ee

\end{lemma}

\begin{proof} After differentiating $(\ref{a1})_2$ with respect to $t$ twice, we get the following equality
\be\la{cxb59}\ba
&\n u_{ttt}+\n u\cdot\na u_{tt}-(\lambda+2\mu)\nabla{\rm div}u_{tt}+\mu\nabla^{\perp}\curl u_{tt}\\
&= 2{\rm div}(\n u)u_{tt}
+{\rm div}(\n u)_{t}u_t-2(\n u)_t\cdot\na u_t-(\n_{tt} u+2\n_t u_t)
\cdot\na u\\& \quad- \n u_{tt}\cdot\na u-\na P_{tt}.
 \ea\ee

Then, multiplying (\ref{cxb59}) by $2u_{tt}$ and   integrating over $\Omega ,$ it is easy to get
\be \la{cxb60}\ba
&\frac{d}{dt}\int_{ }\n
|u_{tt}|^2dx+2(\lambda+2\mu)\int_{ }(\div u_{tt})^2dx+2\mu\int_{ }(\curl u_{tt})^2dx \\
&\quad -2\mu\int_{\partial\Omega}(\kappa-\vartheta) |u_{tt}|^2ds\\
&=-8\int_{ }  \n u^i_{tt} u\cdot\na
 u^i_{tt} dx-2\int_{ }(\n u)_t\cdot \left[\na (u_t\cdot u_{tt})+2\na
u_t\cdot u_{tt}\right]dx\\&\quad -2\int_{
}(\n_{tt}u+2\n_tu_t)\cdot\na u\cdot u_{tt}dx-2\int_{ }   \n
u_{tt}\cdot\na u\cdot  u_{tt} dx\\&\quad+2\int_{ } P_{tt}{\rm
div}u_{tt}dx\triangleq\sum_{i=1}^5J_i.
\ea\ee

 First, H\"{o}lder's inequality shows
\be \la{cxb61} \ba |J_1|&\le
C\|\n^{1/2}u_{tt}\|_{L^2}\|\na u_{tt}\|_{L^2}\| u \|_{L^\infty}\\
&\le \de \|\na u_{tt}\|_{L^2}^2+C(\de)\|\n^{1/2}u_{tt}\|^2_{L^2} .
\ea\ee
It follows from (\ref{cxb17}), (\ref{cxb24}), (\ref{cxb25}), and
(\ref{cxb2}) that
\be \la{cxb62}\ba
|J_2|&\le C\left(\|\n
u_t\|_{L^3}+\|\n_t u\|_{L^3}\right)\left(\| u_{tt}\|_{L^6}\| \na
u_t\|_{L^2}+\| \na u_{tt}\|_{L^2}\| u_t\|_{L^6}\right)\\&\le
C\left(\|\n^{1/2} u_t\|^{1/2}_{L^2}\|u_t\|^{1/2}_{L^6}+\|\n_t
\|_{L^6}\| u\|_{L^6}\right)  \| \na u_{tt}\|_{L^2}  \| \na u_{t}\|_{L^2} \\ &\le \de
\|\na u_{tt}\|_{L^2}^2+C(\de)\si^{-3/2},\ea\ee

\be  \la{cxb63}\ba |J_3|&\le C\left(\|\n_{tt}\|_{L^2}
\|u\|_{L^\infty}\|\na u\|_{L^3}+\|\n_{
t}\|_{L^6}\|u_{t}\|_{L^6}\|\na u \|_{L^2}\right)\|u_{tt}\|_{L^6} \\
&\le \de \|\na u_{tt}\|_{L^2}^2+C(\de)\|\n_{tt}\|_{L^2}^2+C(\de)\si^{-1},
\ea\ee
and
\be  \la{cxb64}\ba
|J_4|+|J_5|&\le C\|\n u_{tt}\|_{L^2} \|\na
u\|_{L^3}\|u_{tt}\|_{L^6} +C \|P_{tt}\|_{L^2}\|\na
u_{tt}\|_{L^2}\\
&\le \de \|\na u_{tt}\|_{L^2}^2+C(\de)\|\n^{1/2}u_{tt}\|^2_{L^2}
+C(\de)\|P_{tt}\|^2_{L^2}.
\ea\ee
Choosing $\de$ small enough, combining all these computations we obtain:
\be  \la{cxb66}\ba
&\frac{d}{dt}\|\n^{1/2}u_{tt}\|^2_{L^2}+\|\na u_{tt}\|_{L^2}^2\\
&\le C (\|\n^{1/2}u_{tt}\|^2_{L^2}+C\|\n_{tt}\|^2_{L^2}+C\|P_{tt}\|^2_{L^2})+C \si^{-3/2},
 \ea\ee
 duo to
 \be  \la{cxb65}\ba
\|\nabla u_{tt}\|_{L^2}\leq C(\|\div u_{tt}\|_{L^2}+\|\curl u_{tt}\|_{L^2}) ,
\ea\ee
since $u_{tt}\cdot n=0$ on $\partial\Omega.$
\eqref{cxb66}, (\ref{cxb24}), (\ref{cxb25}) and Gronwall's inequality show that
\be  \la{cxb67}\ba
\sup_{0\le t\le T}\si^2\|\n^{1/2}u_{tt}\|_{L^2}^2+\int_{0}^T\si^2\|\nabla u_{tt}\|_{L^2}^2dt\le C.
\ea\ee
Moreover, by \eqref{cxb48} and \eqref{cxb25}

\be  \la{cxb68}\ba
\sup_{0\le t\le T}\si^2\|\nabla u_t\|_{H^1}^2\le C.
\ea\ee

Finally, combination of \eqref{cxb52}, \eqref{cxb25}, \eqref{cxb44}, \eqref{cxb45}, \eqref{cxb67}, and \eqref{cxb68},
 we obtain
\be  \la{cxb69}\ba
\si\|\na^2 u\|_{W^{1,q}}
& \le C(\si +\si\|\na  u_t\|_{L^2}+\si\| \na(\n
\dot u )\|_{L^{q}}+\si\|\na^2  P\|_{L^{q}})\\
& \le C(1 +  \si\|\na u\|_{H^2}+\sigma^{1/2}(\sigma\|\na u_t\|_{H^1}^2)^{(q-2)/2q}+\sigma\|\nabla^{2}P\|_{L^q})\\
&\le C(1+\sigma^{1/2}(\sigma^{-1})^{(q-2)/2q})\\
&\le C.
\ea\ee
Together with (\ref{cxb67}) and (\ref{cxb68}) yields (\ref{cxb58}) and concluding the proof of Lemma \ref{xle5}.
\end{proof}

\section{\la{se6}Proofs of  Theorems  \ref{th1}}

Through the previous proofs, we've gotten all the a priori estimates what we need. In this section, we will prove the main results of this paper.

{\it Proof of Theorem \ref{th1}.} By Lemma \ref{loc1}, there is a
$T_*>0$ such that the Cauchy problem \eqref{a1}--\eqref{ch1} has a unique classical solution $(\rho,u)$ on $\Omega\times
(0,T_*]$. One will use the a priori estimates, Proposition \ref{pr1} and Lemmas \ref{xle4}-\ref{xle5} to extend the local classical
solution $(\rho,u)$ to be a global one.

First, it is easy to check that
$$ A_1(0)+A_2(0)=0,\,\,  0\leq\rho_0\leq \hat{\rho}.$$
Therefore, there exists a
$T_1\in(0,T_*]$ such that
\be\la{dlbh1}\ba
0\leq\rho\leq2\hat{\rho},\,\,  A_1(T)+A_2(T)\leq 2C_0^{\frac{1}{3}} ,
\ea\ee
hold for $T=T_1.$

Set
\bn \la{dlbh2}
T^*=\sup\{T\,|\,{\rm (\ref{dlbh1}) \ holds}\}.
\en
Obviously, $T^*\geq T_1>0$. Consequently, for any $0<\tau<T\leq T^*$
with $T$ finite, it follows from Lemmas \ref{xle4} and \ref{xle5}
that
 \be \la{dlbh3}\begin{cases}
   \rho-\bar{\rho} \in C([0,T];W^{2,q}), \\ \na u_t \in C([\tau ,T];L^q),\quad
 \na u,\na^2u \in C\left([\tau ,T];
 C (\bar{\Omega})\right),\end{cases}\ee
 where one has used the standard
embedding
$$L^\infty(\tau ,T;H^1)\cap H^1(\tau ,T;H^{-1})\hookrightarrow
C\left([\tau ,T];L^q\right),\quad\mbox{ for any } q\in [2,\infty).  $$
Due to (\ref{cxb17}), (\ref{cxb25}), (\ref{cxb58}), and $(\ref{a1})_2 $
we obtain
\be\ba
&\int_{\tau}^T \left|\left(\int\n|u_t|^2dx\right)_t\right|dt\no
&\le\int_{\tau}^T\left(\|  \n_t  |u_t|^2 \|_{L^1}+2\|  \n  u_t\cdot u_{tt} \|_{L^1}\right)dt\\
&\le C\int_{\tau}^T \left( \| \n|\div u||u_t|^2 \|_{L^2}+\|  |u||\na \n| |u_t|^2 \|_{L^1}+ \|\rho^{\frac{1}{2}}  u_t
\|_{L^2}\|\rho^{\frac{1}{2}}u_{tt} \|_{L^2}\right)dt\\
&\le C\int_{\tau}^T\left( \| \n^{\frac{1}{2}} |u_t|^2 \|_{L^2}\|\na u\|_{L^\infty}+\|  u\|_{L^6}\|\na\n\|_{L^2} \|u_t  \|^2_{L^6}+  \|\rho^{\frac{1}{2}}u_{tt} \|_{L^2}\right)dt\\
&\le C,\ea\ee
which along with \eqref{dlbh3} yields
\be\la{dlbh4} \|\rho^{1/2}u_t\|_{L^2}, \quad\|\rho^{1/2}\dot u\|_{L^2}\in C([\tau,T],L^2).\ee

Finally, we assert that \be \la{dlbh5}T^*=\infty.\ee Otherwise,
$T^*<\infty$. Then by Proposition \ref{pr1}, it holds that
\be\la{dlbh6}\ba
0\leq\rho\leq\frac{7}{4}\hat{\rho} ,\,\,\,A_1(T)+A_2(T)\leq C_0^{\frac{1}{3}} .
\ea\ee
$(\n(x,T^*),u(x,T^*))$ meets
the initial data condition (\ref{dt1}) and (\ref{dt2}) except $ u(\cdot,T^*)\in H^s$, due to (\ref{dlbh4}), Lemmas \ref{xle4} and \ref{xle5}, where  $g(x)\triangleq\n^{1/2}\dot u(x, T^*),\,\,x\in \Omega.$ Therefore, Lemma
\ref{loc1} shows that there exists some $T^{**}>T^*$, such that
(\ref{dlbh1}) holds for $T=T^{**}$, which contradicts the definition of $ T^*.$
As a result, $0<T<T^*=\infty$. By Lemma \ref{loc1},\eqref{xle4}, \eqref{xle5} and \eqref{dlbh3} indicates that $(\rho,u)$ is in fact the unique classical solution defined on $\Omega\times(0,T]$ for any  $0<T<T^*=\infty.$

It remains to prove \eqref{qa1w}. Integrating $\eqref{a1}_1$ over $\O\times (0,T)$ and using \eqref{ch1} yields that \be \la{bz11}\bar\n=\frac{1}{|\O|}\int\n (x,t)dx\equiv \frac{1}{|\O|} \int \n_0dx. \ee

For $G(\rho)$, there exists a suitably small positive constant $\tilde{C} <1$ depending only on $a,\,\gamma,\,\bar{\rho}_0,$ and $\hat \n$ such that for any $\rho\in [0,2\hat\n]$,
\be\label{gine1}  \tilde{C}^2( \rho-\bar{\rho})^2\le \tilde{C}  G(\rho)  \leq    (\rho^\gamma-\bar{\rho}^\gamma)( \rho - \bar{\rho}). \ee

Multiplying $\eqref{a1}_2$ by $\mathcal{B}[\n-\bar\n]$, we get
\be\la{c59} \ba &
\int(P-P(\bar\n))(\n-\bar\n) dx \\&= \left(\int\rho u\cdot\mathcal{B}[\n-\bar\n] dx\right)_t-\int\rho u\cdot\nabla\mathcal{B}[\n-\bar\n]\cdot udx - \int\rho u\cdot\mathcal{B}[\n_t]  dx \\
& \quad  +\mu\int\nabla u\cdot\nabla\mathcal{B}[\n-\bar\n] dx + (\lambda+\mu)\int(\rho-\bar{\rho})\div udx \\
& \leq \left(\int\rho u\cdot\mathcal{B}[\n-\bar\n] dx\right)_t+C(\hat{\rho})\|u\|_{L^{4}}^{2}\|\n-\bar\n\|_{L^2} +C(\hat{\rho})\|\rho u\|_{L^2}^2\\
& \quad  +C\|\rho-\bar{\rho}\|_{L^2}\|\nabla u\|_{L^2} \\
& \leq \left(\int\rho u\cdot\mathcal{B}[\n-\bar\n] dx\right)_t+\de \|\n-\bar\n\|_{L^2}^2 +C(\de)\|\na u\|_{L^2}^2,
\ea\ee
which, along with \eqref{gine1} and \eqref{tdu1}, leads to
\be \la{gine2} \ba
a\tilde{C}\int G(\rho)dx&\leq a\int(\rho^\gamma-\bar{\rho}^\gamma)( \rho - \bar{\rho})dx\\
&\leq 2\left(\int\rho u\cdot\mathcal{B}[\n-\bar\n] dx\right)_t+C_1\phi(t),
\ea\ee
where $\phi(t)\triangleq (\lambda+2\mu)\|\div u \|_{L^{2}}^{2}+\mu\|\curl u\|_{L^{2}}^{2}.$

Moreover, it follows from \eqref{gine1} and Young's inequality that
\be \la{c511} \ba
\left|\int\rho u\cdot\mathcal{B}[\n-\bar\n] dx\right|\leq C_2\left(\frac{1}{2}\|\sqrt{\rho} u\|^2_{L^2}+\int G(\rho)dx\right),
\ea\ee
 which gives
\be \la{c512} \ba
\frac{1}{2}\left(\frac{1}{2}\|\sqrt{\rho} u\|^2_{L^2}+\int G(\rho)dx\right)\leq W(t)\leq \frac{3}{2}\left(\frac{1}{2}\|\sqrt{\rho} u\|^2_{L^2}+\int G(\rho)dx\right),\ea\ee
 where $$W(t)=\int \left(\frac{1}{2}\rho |u|^2+G(\rho)\right)dx-\delta_0\int\rho u\cdot\mathcal{B}[\n-\bar\n] dx,$$ with $\delta_0=\min\{\frac{1}{2C_1},\frac{1}{2C_2}\}.$

Adding \eqref{gine2} multiplied by $\de_0 $ to \eqref{m9}
and using \bnn \int\n |u|^2dx\leq C(\hat{\rho})\|\na u\|_{L^2}^2\leq C_3\phi(t),\enn we obtain for a suitably small constant $\delta_1=\delta_1(a,\delta_0, \tilde{C}_0,C_3)$,
 \bnn W'(t)+2\delta_1W(t)\leq 0,\enn which together with \eqref{c512} yields that for any $t\geq0$,
\begin{equation}\label{c513}
\int\left(\frac{1}{2}\rho|u|^2+G(\rho)\right)dx\leq 4C_0e^{-2\delta_1t}.
\end{equation}
By \eqref{m9}, we have
\be \la{c514} \ba
\int_0^\infty\phi(t)e^{\delta_1 t} dt\leq C.
\ea\ee

Choosing $m=0$ in \eqref{I0}, by \eqref{I1}, \eqref{I21} and \eqref{I3}, we obtain
\be\la{c515}\ba
&\left(\phi(t)-2\int(P-P(\bar{\rho}))\,\div udx-2\mu\int_{\partial\Omega}(\kappa-\vartheta) |u|^2ds\color{black}\right)_{t}+\frac{1}{2}\|\sqrt{\rho}\dot{u}\|^2_{L^2}\\
&\leq C\left(\phi(t)+\int G(\rho)dx\right).
\ea \ee
Notice that $$\|P-\bar{P}\|_{L^2}^2\leq C\|P-P(\bar{\rho})\|_{L^2}^2\leq C\int G(\rho)dx.$$
Multiplying \eqref{c515} by $e^{\delta_1 t}$, and using the fact
\bnn\ba
\left|\int(P-P(\bar{\rho})\,\div udx\right|\leq C\int G(\rho)dx+\frac{1}{4}\phi(t)
\ea\enn
and
\bnn\ba
\left|\int_{\partial\Omega}(\kappa-\vartheta) |u|^2ds\right|\color{black}\leq C \phi(t),
\ea\enn
we get
\be\la{c516}\ba
&\left(e^{\delta_1 t}\phi(t)-2e^{\delta_1 t}\int(P-P(\bar{\rho})\,\div udx-2e^{\delta_1 t}\mu\int_{\partial\Omega}(\kappa-\vartheta) |u|^2ds\color{black}\right)_{t}+\frac{1}{2}e^{\delta_1 t}\|\sqrt{\rho}\dot{u}\|^2_{L^2}\\
&\leq Ce^{\delta_1 t}\left(\phi(t)+\int G(\rho)dx\right),
\ea \ee
which, together with \eqref{c513} and \eqref{c514}, yields that for any $t>0$,
\be\la{c517}\ba
\|\nabla u\|_{L^2}^2\leq Ce^{-\delta_1 t},
\ea \ee
and
\be\la{c518}\ba
\int_0^\infty e^{\delta_1 t}\|\sqrt{\rho}\dot{u}\|^2_{L^2}dt\leq C.
\ea \ee
By \eqref{ax401}, \eqref{h99}, \eqref{c517} and \eqref{c518}, a direct calculation leads to
\be\la{c519}\ba
\|\sqrt{\rho}\dot{u}\|^2_{L^2}\leq Ce^{-\delta_1 t},\,\,t\geq1.
\ea \ee

Finally, together with \eqref{c513}, \eqref{c517}, \eqref{c519} and \eqref{h18}, we obtain \eqref{qa1w} for some positive constant $\tilde{\eta}\leq\delta_1$ depending only on $\mu,$  $\lambda,$  $\gamma,$ $a$, $s$, $\on$, $\hat{\rho}$, $M$, $\Omega$, $p$, $r$ and $C_0$ and  finish the proof.
\thatsall

{\it Proof of Theorem \ref{th2}.}
Suppose $T>0$, we introduce the Lagrangian coordinates
  \be \la{c61}  \begin{cases}\frac{\partial}{\partial \tau}X(\tau; t,x) =u(X(\tau; t,x),\tau),\,\,\,\, &0\leq \tau\leq T,\\
 X(t;t,x)=x, \,\,\,\, &0\leq t\leq T,\,x\in\bar{\Omega}.\end{cases}\ee
By \eqref{dt6}, it is easy to find that \eqref{c61} is well-defined. \eqref{c61} together with $\eqref{a1}_1$ shows
 \be\la{c62}\ba
\rho(x,t)=\rho_0(X(0; t, x)) \exp \left\{-\int_0^t\div u(X(\tau;t, x),\tau)d\tau\right\}.
\ea \ee

If there exists some point $x_0\in \Omega$ such that $\n_0(x_0)=0,$ then for any $t>0$, $X(0; t, x_0(t))=x_0$. Hence, for any $t\geq 0,$ $\rho(x_0(t),t)\equiv 0$ due to \eqref{c62}. As a result, Gagliardo-Nirenberg's inequality shows that for any $\tilde{q}\in(1,\infty)$ and $ \tilde{r}\in  (2,\infty),$
\be\la{c63}\ba\bar{\rho}_0\equiv\bar\n\leq\|\rho-\bar{\rho}\|_{C\left(\ol{\O }\right)} \le C
\|\rho-\bar{\rho}\|_{L^{\tilde{q}}}^{\tilde{\theta}}\|\na \rho\|_{L^{\tilde{r}}}^{1-\tilde{\theta}}
\ea\ee
where $\tilde{\theta}=\tilde{q}(\tilde{r}-2)/(2\tilde{r}+\tilde{q}(\tilde{r}-2))$. Together with \eqref{qa1w}, we gives \eqref{qa2w}. This completes the proof.
\thatsall




\section{Acknowledgements}

The author would like to thank Professors Jing Li and Guocai Cai for their valuable discussions. The
research is partially supported by National Natural Science Foundation of China (Nos. 11961068, 11761058).







\end{document}